\documentclass[leqno,12pt]{amsart} 
\usepackage{amsmath}
\usepackage{amssymb}
\usepackage{amscd}
\usepackage{mathrsfs}
\usepackage[dvipdfmx]{graphicx}
\usepackage{graphics}
\usepackage{url}
\usepackage[usenames]{color}

\theoremstyle{plain} 
\newtheorem{theorem}{\indent\sc Theorem}[section]
\newtheorem{lemma}[theorem]{\indent\sc Lemma}
\newtheorem{corollary}[theorem]{\indent\sc Corollary}
\newtheorem{proposition}[theorem]{\indent\sc Proposition}

\newtheorem{conjecture}[theorem]{\indent\sc Conjecture}
\theoremstyle{definition} 

\newtheorem{remark}[theorem]{\indent\sc Remark}

\newtheorem{question}[theorem]{\indent\sc Question}

\begin{document}

\title[On the neighborhood of a torus leaf and dynamics]{On the neighborhood of a torus leaf and dynamics of holomorphic foliations}

\author[Takayuki KOIKE and Noboru OGAWA]{Takayuki KOIKE$^{1}$ and Noboru OGAWA$^{2}$}
\address{ 
$^{1}$ Department of Mathematics \\
Graduate School of Science \\
Osaka Metropolitan University \\
3-3-138 Sugimoto, Sumiyoshi-ku, \\
Osaka 558-8585 \\
Japan 
}
\email{tkoike@omu.ac.jp}
\address{
$^{2}$ Department of Mathematics \\
Tokai University  \\
4-1-1 Kitakaname, Hiratsuka-shi,\\
Kanagawa 259-1292 \\
Japan
}
\email{nogawa@tokai.ac.jp}
\thanks{The first author is supported by Leading Initiative for Excellent Young Researchers (No. J171000201). The second author is partially supported by JSPS KAKENHI Grant Number 17K05283, 21K03256. 
They are also partly supported by Osaka City University Advanced Mathematical Institute (MEXT Joint Usage/Research Center on Mathematics and Theoretical Physics). }

\begin{abstract}
Let $X$ be a complex surface and $Y$ be an elliptic curve embedded in $X$. 
Assume that there exists a non-singular holomorphic foliation $\mathcal{F}$ with $Y$ 
as a compact leaf, defined on a neighborhood of $Y$ in $X$. 
We investigate the relation between Ueda's classification of the complex analytic structure of a neighborhood of $Y$ and 
complex dynamics of the holonomy of $\mathcal{F}$ along $Y$. 
More precisely, we show that the pair $(Y,X)$ is of type ($\gamma$) in his classification when there exists a closed curve in $Y$ 
along which the holonomy of $\mathcal{F}$ is irrationally indifferent and non-linearizable. 
We also investigate the metric semi-positivity of the line bundle determined by the divisor $Y$. 
Our approach is based on the theory of hedgehogs, due to P\'{e}rez-Marco. 
\end{abstract}

\maketitle

\section{Introduction}

Let $X$ be a complex manifold and $Y$ be a compact complex submanifold of $X$. 
Assume that there exists a non-singular holomorphic foliation $\mathcal{F}$ 
which has $Y$ as a compact leaf and is defined on a neighborhood $V$ of $Y$ in $X$. 
Our interest is the relation between 
complex-analytic properties of small neighborhoods of $Y$ and 
complex dynamics of 
the holonomy of $\mathcal{F}$ along $Y$. 

In this paper, we consider the case where $X$ is a complex surface and $Y$ is a compact Riemann surface (mainly an elliptic curve) as a compact leaf of $\mathcal{F}$.  
It follows from the fundamental results in foliation theory that the normal bundle $N_{Y/X}$ of $Y$ is topologically trivial;   
Since $N_{Y/X}$ is a holomorphic line bundle on a compact Riemann surface $Y$ and its degree ${\rm deg}\,N_{Y/X} (=c_1(N_{Y/X})=(Y^2))$ is zero by the existence of the Bott connection, which is flat along $Y$  (\cite{B} or \cite{CS}, see also \cite[p. 37]{Br}). 
On the other hand, 
the complex analytic structure on a neighborhood of $Y$ may behave complicatedly even if the normal bundle is holomorphically trivial. 
The neighborhood of an embedded curve with topologically trivial normal bundle is analytically classified into three types $(\alpha), (\beta)$, 
and $(\gamma)$ according to Ueda's classification (\cite{U}, see also \S \ref{section:review_uedatheory} below. 
Original Ueda's neighborhood theory assumes neither that $Y$ is an elliptic curve nor the existence of a foliation). 
A pair $(Y, X)$ is said to be of type $(\beta)$ if there exists a non-singular holomorphic foliation $\mathcal{G}$ 
defined on a neighborhood of $Y$ which has $Y$ as a leaf and has $\textrm{U}(1)$-linear holonomy along $Y$. 
In this case, $Y$ admits a system of \textit{pseudoflat} neighborhoods. 
A pair $(Y, X)$ is said to be of type $(\alpha)$ if, roughly speaking, 
it is different from the case of type $(\beta)$ in $n$-jet along $Y$ for some positive integer $n$ (see \S \ref{section:review_uedatheory} for the precise definition). 
Ueda showed that $Y$ admits a system of \textit{strongly pseudoconcave} neighborhoods in this case \cite[Theorem 1]{U}. 
The remaining case is called type $(\gamma)$. 
So far, little is known about this case. 
The first example of the pair $(Y, X)$ of type $(\gamma)$ is constructed by Ueda \cite[\S 5.4]{U}, in which $Y$ is a compact Riemann surface of genus larger than $0$. 
Note that this example (of Ueda's) satisfies our assumption on the existence of $\mathcal{F}$ on $V$ as above. 
In \cite{K}, the first author investigated the complex-analytic properties of small neighborhoods of $Y$ for Ueda's example $(Y, X)$ when $Y$ is an elliptic curve, 
whose generalization is one of the biggest motivations of the present paper. 
Our purpose is to determine the type of $(Y,X)$ from the information about the holonomy of $\mathcal{F}$ along $Y$. 
In what follows we always assume that $Y$ is an elliptic curve. 

To state our main result, we recall the notion of holonomy of $\mathcal{F}$. 
Take a closed smooth curve $c \colon [0,1]\to Y$ with $p=c(0)=c(1)$ and a transversal $\tau$ for $\mathcal{F}$ through $p$, i.e. holomorphically embedded open disk of $\text{dim}_{\mathbb{C}}=1$ which intersects the leaves of $\mathcal{F}$ transversely. 
If $\tau$ is contained in a foliated chart $(W;(z,w))$, $\tau$ is parametrized by the transverse coordinate $w$.  
We cover $c$ by finitely many foliated charts $\{(W_{j};(z_{j},w_{j}))\}_{j}$ such that $W_{j}\cap Y=\{w_{j}=0\}$, $j=1,\dots,N$, and $W_{1}$ contains $\tau$. 
By moving from one chart to the next chart, we have a local holomorphic diffeomorphism on the $w$-coordinate. 
If we go around along $c$, the composition of the finitely many transition maps gives rise to a local holomorphic diffeomorphism  on $\tau$. 
Upon fixing a parametrization of $\tau$ by $w_{1}$, the composition determines an element $h_{\mathcal{F}}(c, \tau, \{W_{j}\})$ of $\text{Diff}(\mathbb{C},0)$. 
Here, ${\rm Diff}(\mathbb{C}, 0)$ stands for the group of germs at $0$ of local holomorphic diffeomorphisms on $\mathbb{C}$ fixing $0$. 
The germ $h_{\mathcal{F}}(c, \tau, \{W_{j}\})$ is called the \textit{holonomy} of $\mathcal{F}$ along $c$. 
This is not changed under a leafwise homotopy, keeping $p$ fixed, of $c$ inside the charts. 
Moreover, it does not depend on the choice of foliated charts by the cocycle condition on $w_{j}$'s. 
Thus, the germ depends only on its leafwise homotopy class $\gamma:=\langle c \rangle$ keeping $p$ fixed. 
The variation of the holonomy under the choice of transversals can be expressed as the conjugation of the elements in $\text{Diff}(\mathbb{C},0)$, so that the \textit{holonomy homomorphism} $h_{\mathcal{F}}\colon \pi_{1}(Y,p)\to \text{Diff}(\mathbb{C},0)$ is defined, up to conjugation.

We say the holonomy of $\mathcal{F}$ along $c$ is \textit{linearizable} (resp. \textit{non-linearizable}) if the corresponding map $h_{\mathcal{F}}(\gamma)$ is linearizable (resp. non-linearizable) at $0$, where $\gamma:=\langle c \rangle \in \pi_{1}(Y, p)$. 
Also, the holonomy is said to be \textit{rationally} (resp. \textit{irrationally}) \textit{indifferent} if the fixed point $0$ of $h_{\mathcal{F}}(\gamma)$ is rationally (resp. irrationally) indifferent. 
See \S 2 for the definitions. 
Note that, these properties of $h_{\mathcal{F}}(\gamma)$ are invariant under conjugation in $\text{Diff}(\mathbb{C},0)$, so that the definitions do not depend on the choice of transversals. 
The main result is the following. 

\begin{theorem}\label{thm_maincor}
Let $X$ be a complex surface and $Y$ be an elliptic curve embedded in $X$. 
Assume that there exist a non-singular holomorphic foliation $\mathcal{F}$ with $Y$ 
as a compact leaf, defined on a neighborhood of $Y$ in $X$, and a closed curve $c$ in $Y$ 
along which the holonomy of $\mathcal{F}$ is irrationally indifferent and non-linearizable. 
Then, $(Y,X)$ is of type ($\gamma$). 
\end{theorem} 

In order to show that the pair $(Y, X)$ as in Theorem \ref{thm_maincor} is of type ($\gamma$) by contradiction, 
we need to compare the foliation $\mathcal{F}$ with the foliation $\mathcal{G}$ as in the definition of type $(\beta)$ above in some sense. 
Such type of arguments can be regarded as an analytic variant of the argument in \cite[\S 4]{LTT}, 
in which the formal comparison of two (formal) foliations are investigated. 
To prove Theorem \ref{thm_maincor}, it is necessary  to discuss not only such formal arguments but also some arguments on analytic foliations. 
In this paper, we address the problem from the perspective of complex dynamics. 
Our approach is based on the theory of hedgehogs, due to P\'{e}rez-Marco. 

This result includes Ueda's example, see \S 3.3.2. 
Ueda used the small cycle property to show that it is of type $(\gamma)$. 
Comparing with this, we will use a complete invariant set around irrationally indifferent fixed point $0$, the so-called \textit{hedgehog}, which was introduced by P\'{e}rez-Marco \cite{P4}. 
This result also includes the example of \cite[\S 5.4]{GS} as the case IV in our classification in \S 3.1. Our strategy of the proof of Theorem \ref{thm_maincor} gives a geometric interpretation for the arguments in \cite[\S 5.4]{GS}. 

By using Theorem \ref{thm_maincor}, one can generalize \cite[Theorem 1.1]{K}. 
Although we will give the precise statement in \S 3.1, 
we restrict here it to the formulation of \cite{K} to compare it with the previous theorem. 
Consider the following two conditions: 
$(1)$ there exists a neighborhood $V$ of $Y$ and a holomorphic submersion $\pi\colon V\to Y$ such that $\pi|_Y$ is the identity map, and 
$(2)$ there exists a non-trivial element $\gamma_1 \in \pi_1(Y, p)$ such that $h_{\mathcal{F}}(\gamma_1)$ is the identity map.
Then, as a generalization of \cite[Theorem 1.1]{K}, we have the following statement. 
\begin{theorem}\label{thm:main}
Assume that $(Y, X; \mathcal{F})$ satisfies the conditions $(1)$ and $(2)$ above. 
Denote by $g:=h_{\mathcal{F}}(\gamma_2)$ where $\pi_1(Y, p) \cong\mathbb{Z}\gamma_1 \oplus \mathbb{Z} \gamma_2$. 
Then, the following hold: 
\begin{enumerate}
\item[(i)] The pair $(Y, X)$ is of type $(\alpha)$ if and only if 
$g$ has a rationally indifferent fixed point $0$ and is not of finite order, i.e. $g^n\neq {\rm id}$ for any $n>0$. 
\item[(ii)]  The pair $(Y, X)$ is of type $(\beta)$ if and only if $g$ is linearizable at $0$. 
\item[(iii)]  The pair $(Y, X)$ is of type $(\gamma)$ if and only if $g$ has an irrationally indifferent fixed point $0$ and is non-linearizable at $0$. 
\end{enumerate}
\end{theorem}
See \S \ref{section:examples} for the result in more generalized configurations and their proofs. 

We here describe the outline of the proof of Theorem \ref{thm_maincor}, see \S 4.2 for more details. 
First, we show that the pair $(Y, X)$ is not of type $(\alpha)$ 
by a standard observation of the normal bundle of $Y$. 
Second, we show that the pair $(Y, X)$ is not of type $(\beta)$ by contradiction. 
If the pair $(Y, X)$ is of type $(\beta)$, there exists a pluriharmonic function $\Phi\colon V\setminus Y\to \mathbb{R}$ 
for a sufficiently small neighborhood $V$ of $Y$ such that $\Phi(p)=O(-\log {\rm dist}(p, Y))$ as $p\to Y$, where ``${\rm dist}$'' is the local Euclidean distance. 
This contradicts Theorem \ref{thm_ph} below. 
Hence, Theorem \ref{thm_maincor} is deduced from Theorem \ref{thm_ph}. 

\begin{theorem}\label{thm_ph}
Let $(Y, X; \mathcal{F})$ be as in Theorem \ref{thm_maincor}, $V$ a neighborhood of $Y$ in $X$ and $\Phi \colon V \to \mathbb{R} \cup \{\infty\}$ a continuous function which is pluriharmonic on $V\setminus Y$, where $\mathbb{R}\cup\{\infty\}$ is homeomorphic to the standard $(0,1]\subset \mathbb{R}$. 
Then, $\Phi$ is bounded from above on a neighborhood of $Y$. 
\end{theorem}

Note that the assumptions of Theorem \ref{thm_ph} imply that $\Phi$ is automatically bounded from below in a neighborhood of $Y$. 
Theorem \ref{thm_ph} is shown by contradicting to the maximum principle for the restriction of $\Phi$ 
on a dense leaf $L$ of $\mathcal{F}$ in an invariant set containing $Y$.  
The existence of such an invariant set is guaranteed by Theorem \ref{thm:hedgehog}, see \S 4.1 for more details. 

\begin{theorem}[{\cite{P2}, \cite{P4}, \cite{P5}}]\label{thm:hedgehog}
Let $f(z)=e^{2\pi\sqrt{-1}\theta}z+O(z^2)$ and $g(z)=e^{2\pi\sqrt{-1}t}z+O(z^2)$, $\theta, t\in \mathbb{R}$, be local holomorphic diffeomorphisms which satisfy $f \circ g = g\circ f$. 
Assume that $t$ is an irrational number. 
Then, the following statements hold. 
\begin{enumerate}
\item[(i)] For arbitrary small neighborhood $U\subset \mathbb{C}$ of $0$, there exists a compact connected subset $K$ of $\overline{U}$ 
such that $0\in K$, $K\not=\{0\}$, $\mathbb{C}\setminus K$ is connected, 
and that $K$ is completely invariant under $f$ and $g$, i.e. $f(K)=f^{-1}(K)=g(K)=g^{-1}(K)=K$. 
\item[(ii)] If $g$ is linearizable at $0$, then $f$ is also linearizable by the linearization map of $g$. In this case, $K$ can be chosen as the closure of a domain in $U$ with $0\in \textrm{Int}\hspace{0.5mm}(K)$. 
\item[(iii)] If $g$ is non-linearizable at $0$, then $K$ contains $0$ as the boundary point. 
\item[(iv)] There exists a point $x_{0}$ of $\partial K$ such that the orbit $(g^n(x_{0}))_{n\in\mathbb{Z}}$ is dense in $\partial K$. 
\end{enumerate}
\end{theorem}

Theorem \ref{thm:hedgehog} is a combination of several facts in the theory of commuting holomorphic germs fixing the origin $0$. See the end of \S2.1.

By imposing an additional condition on the holonomy of $\mathcal{F}$, we have the statement about \textit{plurisubharmonic} functions. Compare with Theorem \ref{thm_ph}. 

\begin{theorem}\label{thm_psh}
Let $(Y, X; \mathcal{F})$ be as in Theorem \ref{thm_maincor}, $V$ a neighborhood of $Y$ in $X$ and $\Phi\colon V \to \mathbb{R}\cup\{\infty\}$ a continuous function which is plurisubharmonic on $V\setminus Y$, where $\mathbb{R}\cup\{\infty\}$ is homeomorphic to the standard $(0, 1]\subset \mathbb{R}$. 
In addition, we assume that there is a homotopically non-trivial closed curve $c'$ in $Y$ along which the holonomy is of finite order. 
Then, $\Phi$ can be extended as a plurisubharmonic function on $V$ which is bounded from above on a neighborhood of $V$. 
\end{theorem}

Theorem \ref{thm_psh} can be regarded as a weak analogue of Ueda's theorem \cite[Theorem 2]{U} 
on the constraint of the increasing degree of plurisubharmonic functions.
In \cite{K2}, the first author applied \cite[Theorem 2]{U} to show the non-semipositivity 
(i.e. non-existence of a $C^\infty$ Hermitian metric with semi-positive curvature) of a line bundle $L$ on $X$ 
which corresponds to the divisor $Y$ when $(Y, X)$ is of type $(\alpha)$. 
As an application of Theorem \ref{thm_psh}, we have the following: 

\begin{corollary}\label{cor:semipositivity}
Assume that $(Y, X; \mathcal{F})$ is as in Theorem \ref{thm:main}. 
Let $L$ be the line bundle on $X$ which corresponds to the divisor $Y$. 
Then $L$ is semi-positive (i.e. $L$ admits a $C^\infty$ Hermitian metric with semi-positive curvature) 
if and only if the pair $(Y, X)$ is of type $(\beta)$. 
\end{corollary}

We believe that Corollary \ref{cor:semipositivity} can be regarded as a supporting evidence of the following conjecture for the line bundle $[Y]$ on $X$ which corresponds to a divisor $Y$. 

\begin{conjecture}[{\cite[Conjecture 1.1]{K4}}]\label{conjecture:main}
Let $X$ be a complex surface and $Y$ be a compact smooth curve holomorphically embedded in $X$ such that the normal bundle is topologically trivial. 
The line bundle $[Y]$ admits a $C^\infty$ Hermitian metric with semi-positive curvature if and only if the pair $(Y, X)$ is of type $(\beta)$. 
\end{conjecture}

\vskip3mm
{\bf Addendum. } Main part of this paper is based on our preprint \cite{KO} uploaded first in 2018. 
Here we report some progress concerning on Conjecture \ref{conjecture:main} after that. 
In 2020, the first author obtained some sufficient conditions for this conjecture to be affirmative \cite{K5}. 
Just after that, Takeo Ohsawa pointed out in \cite[Remark 5.2]{O} that one of our sufficient conditions combined with Siu's solution of Grauert--Riemenschneider conjecture implies that Conjecture \ref{conjecture:main} is affirmative when $X$ is a compact surface. 
In 2021, it has been shown that Conjecture \ref{conjecture:main} is affirmative when $X$ is a compact K\"ahler manifold \cite[Corollary 1.5]{K6}, in the proof of which our technique concerning on the application of the hedgehog theory in the present paper has been applied. Here we emphasize that, in the present paper, we assume neither the compactness nor K\"ahlerity for the manifold $X$. 

\vskip3mm
{\bf Plan of the paper. } The organization of the paper is as follows. 
In \S 2, we review some fundamental results on linearization theorems, Siegel compacta, and Ueda theory.  
In \S 3, we divide the situation into ten cases and state a variant of some of our main results as Theorem \ref{thm:main_oldtype}. 
We prove several cases in this theorem, which directly follow from known results. 
Moreover, we describe typical examples in some cases. 
The main part of this paper is \S4. 
First, we show Theorem \ref{thm_ph}. 
Theorem \ref{thm_maincor} is deduced from this result. 
The proofs of Theorem \ref{thm:main_oldtype} and Theorem \ref{thm:main} are also given here. 
In \S 5, we prove Theorem \ref{thm_psh}. 
In \S 6, we prove Corollary \ref{cor:semipositivity}. 
In Appendix, we discuss the dynamical behavior around non-linearizable hedgehogs and give a slightly modified proof of the existence theorem of common hedgehogs by commuting local holomorphic diffeomorphisms (\cite[Thm.III.14]{P5}). 

\vskip3mm
{\bf Acknowledgment. } 
The authors are grateful to Professor Eric Bedford for informing us about P\'{e}rez-Marco's theory on Siegel compacta. 
We also would like to thank Professor Tetsuo Ueda for valuable comments and suggestions. 
Especially, some of the main theorem could be improved by using Lemma \ref{lem_parabolic}, which we are taught by Professor Tetsuo Ueda.


\section{Preliminaries}
\subsection{Linearization theorems and Siegel compacta} 

In this section, we review some basic properties of linearizability of local holomorphic diffeomorphisms around fixed points. 
Let $f(z)=\lambda z+O(z^2)$ be a local holomorphic diffeomorphism fixing $0$. 
We say that $f$ is \textit{linearizable at $0$} if 
there exist open neighborhoods $U, V$ of $0$, and a biholomorphism $h\colon U\to V$ such that 
$(h\circ f \circ h^{-1}) (z) = \lambda z$. 
It is classically known that the linearizability of $f$ depends on the choice of $\lambda$. 
If $|\lambda|\neq 0,1$, 
then $f$ is linearizable (Koenigs' linearization theorem). 
On the other hand,  there are some obstructions for the case that $|\lambda|=1$. 
The fixed point $0$ is said to be \textit{rationally indifferent} 
(resp. \textit{irrationally indifferent}) 
if $|\lambda|=1$ and $\lambda$ is torsion (resp. non-torsion). 
In the irrationally indifferent case, $f$ is linearizable if the argument $(\log{\lambda})/2\pi \sqrt{-1}$ 
satisfies the Diophantine condition (Siegel's linearization theorem \cite{S}). 
This condition can further be sharpened to the Brjuno condition. 
The maximal linearization domain is biholomorphic to the unit disk, 
on which $f$ is analytically conjugate to a rotation. 
This is called the Siegel disk of $f$. 
Obviously, the domain is completely invariant under $f$, i.e. invariant by both of $f$ and $f^{-1}$. 

On the other hand, in \cite{P4}, P\'{e}rez-Marco showed the existence of completely invariant sets 
(not necessarily linearizable domains) around an indifferent fixed point. 
A Jordan domain $U$ with $C^1$-boundary is said to be \textit{admissible for $f$} if  
$f$ and $f^{-1}$ are defined and univalent on an open neighborhood of the closure $\overline{U}$ of $U$.  
\begin{theorem}[{\cite[Theorem 1.1]{P4}}]\label{siegel} 
Let $f(z)=\lambda z+O(z^2)$ be a local holomorphic diffeomorphism with the indifferent fixed point $0$.  
Let $U$ be an admissible neighborhood of $0$. 
Then there exists a subset $K$ in $\mathbb{C}$ which satisfies the following conditions$\colon$ 
\begin{itemize}
\item[(i)] $K$ is compact, connected, $\mathbb{C}\backslash K$ is connected, 
\item[(ii)] $0\in K \subset \overline{U}$, 
\item[(iii)] $K\cap \partial U \neq \emptyset$, and 
\item[(iv)] $K$ is completely invariant under $f$, i.e. $f(K)=f^{-1}(K)=K$.
\end{itemize}
Moreover, assume that $f$ is not of finite order. 
Then, $f$ is linearizable if and only if $0\in \text{Int}\, K$. 
\end{theorem}\vspace{-5mm}

\begin{figure}[htbp]
 \begin{center}
  \includegraphics[width=80mm]{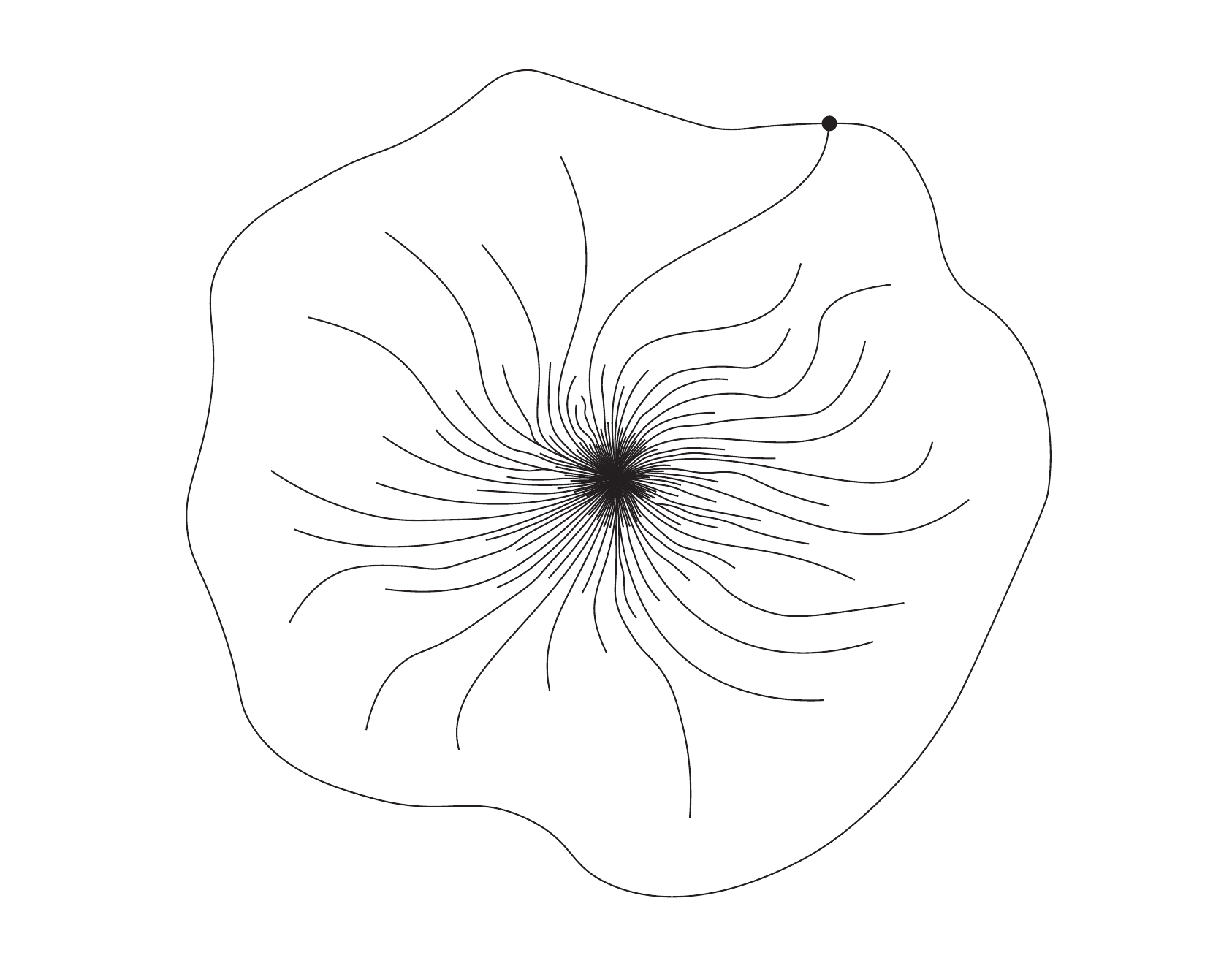}
 \end{center}
  \begin{picture}(0,0)
\put(20,45){\small$K$}
\put(-80,35){\small$U$}
\put(35,180){\small$z_{0}$}
\end{picture}\vspace{-10mm}
 \caption{A hedgehog $K$ of $(U,f)$. }
 \label{fig:one}
\end{figure}

P\'{e}rez-Marco called this completely invariant set \textit{a Siegel compactum}, 
which can be regarded as a degeneration of Siegel disks in some sense. 
The Siegel compactum of the pair $(U,f)$ is denoted by $K_{(U,f)}$. 

\begin{remark}
At least, it is not clear if Theorem \ref{siegel} leads to the uniqueness of the Siegel compactum $K_{(U,f)}$ of the pair $(U,f)$. 
However, there is a canonical choice at least when $U$ is a unit disk. Indeed, there is the maximum of the set of all subset $K$ as in Theorem \ref{siegel}, which is obtained as the connected component of the set 
\[
\{z\in \overline{U}\mid f^n(z)\in \overline{U}~\text{for each integer}\ n\}
\]
which contains the origin $0$. 
\end{remark}

In the irrationally indifferent case, such an invariant set $K$ is called \textit{a hedgehog} 
if it contains a relatively compact linearization domain in $U$ or $f$ is non-linearizable. 
P\'{e}rez-Marco showed the existence of a hedgehog by perturbing the local diffeomorphism $f$ and considering the Hausdorff limit of their Siegel disks. 
According to \cite[Theorem III.8]{P5}, \cite[\S V]{P4}, and \cite[Theorem 5]{P1}, sufficiently near the irrationally indifferent point $0$, 
non-linearizable hedgehogs
have no interior points 
and are not locally connected 
at any point different from $0$. 
By considering the associated analytic circle diffeomorphisms, 
he studied some properties of Siegel compacta and found several applications (see \cite{P4}). 
In particular, with regard to the dynamics on Siegel compacta, 
the following fact is remarkable. 

\begin{theorem}[{\cite[Theorem IV.2.3.]{P4}}]\label{thm:IV.2.3.} 
Let $f$ be as in Theorem \ref{siegel}. 
For $\mu_{K}$-a.e. point $z$ in $K$, the orbit of $z$ is dense in $\partial K$. 
In particular, if $f$ is non-linearizable, the orbit is dense in $K$. 
Here, $\mu_{K}$ is the harmonic measure at $\infty$ of $K$ in $\mathbb{C}P^1$. 
\end{theorem}

We mention some results about commuting holomorphic germs fixing the origin $0$. 

\begin{proposition}[e.g. {\cite[\S I.4]{P2}}]\label{prop:comm}
Let $f(z)=\lambda z+O(z^2)$, $g(z)=\mu z+O(z^2)$ be local holomorphic diffeomorphisms fixing $0$ which satisfy $f\circ g=g\circ f$. 
Assume that $\mu \in \mathbb{C}^{\ast}$ is non-torsion, i.e. $\mu^{m}\neq 1$ for any $m>0$. 
Then the following statements hold. 
\begin{enumerate}
\item If $g$ is linearizable at $0$, 
then $f$ is also linearizable by the linearization map of $g$. 
\item If $g$ is non-linearizable at $0$, then $\lambda=e^{2\pi\sqrt{-1}\theta}$ for some $\theta\in \mathbb{R}$. 
Moreover, \vspace{2pt} 
\begin{enumerate}
\item $f$ is linearizable at $0$ if $\theta \in \mathbb{Q}$, 
\item $f$ is non-linearizable at $0$ if $\theta \in \mathbb{R}\setminus \mathbb{Q}$.  
\end{enumerate}
\end{enumerate}
\end{proposition}

The first part of the statement follows by comparing the coefficients of power series of $f$ and $g$. 
The second part is obtained by using Koenigs' linearization theorem. 
The remaining part also follows by the standard arguments. For more details, see e.g. \cite[\S I.4]{P2}. 

In \cite{P2} and \cite{P5}, 
P\'{e}rez-Marco studied hedgehogs for commuting local holomorphic diffeomorphisms. 
We state a slightly weaker version of his result, which is sufficient for our purpose.  
We give a modified proof in Appendix for the convenience of the readers. 

\begin{theorem}[a part of {\cite[Thm.III.14]{P5}}]\label{Thm.III.14}
Let $f(z)=\lambda z+O(z^2)$ and $g(z)=\mu z+O(z^2)$ be local holomorphic diffeomorphisms 
with an irrationally indifferent fixed point at $0$ which satisfy $f\circ g = g\circ f$.  
Assume that $g$ is non-linearizable at $0$.
Then, for any open neighborhood $W$ of $0$, 
there exists a compact subset $K$ in $W$ 
which is a common hedgehog of $f$ and $g$. 
More precisely, there exist admissible neighborhoods $U$ and $V$ in $W$ 
for $f$ and $g$ respectively such that 
$K_{(U,f)}=K_{(V,g)}$ holds.
\end{theorem}

\begin{proof}[Proof of Theorem \ref{thm:hedgehog}] 
For any neighborhood $V$ of $0$, we take a sufficiently small open neighborhood $U$ which is admissible for $f$ and $g$. 
If necessary, we will replace a smaller neighborhood. 
First, we assume that $g$ is linearizable at $0$. 
Proposition \ref{prop:comm} (1) shows that $f$ is also linearizable by the linearization map of $g$. 
Thus, a Siegel disk $K\subset \overline{U}$ of $g$ is also invariant under $f$. 
The subset $K$ satisfies the statement (i) and (ii). 
Since $g|_{K}$ is conjugate to the irrational rotation, the statement (iv) follows in this case. 

Second, we assume that $g$ is non-linearizable at $0$. 
By Theorem \ref{siegel}, there exists a hedgehog $K\subset \overline{U}$ of $(U,g)$. 
If $\theta$ is a rational number, then by Proposition \ref{prop:comm} (2) (a), $f$ is conjugate to the rational rotation on a domain $W\ni 0$, in particular, $f|_{W}$ is of finite order, i.e. $(f|_{W})^{N}=\textrm{id}$ for some $N>0$. 
As replacing a smaller neighborhood $U$ if necessary, we may assume that $\overline{U}$ is contained in the invariant domain $W$ of $f$. 
Since $f$ and $g$ commute, $f(K)$ is invariant under $g$, and it is a hedgehog of $(f(U),g)$. 
By applying Lemma \ref{lem_a} to $K$ and $f(K)$, we have $K\subset f(K)$ or $f(K)\subset K$. 
We may assume that $f(K)\subset K$ by replacing $f$ with its inverse if necessary. 
Iterating $f$, we have 
\[
K \supset f(K) \supset f^{2}(K) \supset \cdots \supset f^{N}(K)=K
\]
for some $N$, so that $f(K)=f^{-1}(K)=K$. 
Thus, the subset $K$ satisfies the statement (i) and (iii) by Theorem \ref{siegel}. 
If $\theta$ is an irrational number, Theorem \ref{Thm.III.14} show the existence of a common hedgehog  $K'$ of $(U'; f,g)$, $U'\subset U$, which satisfies the statement (i) and (iii). 
The remaining case of the statement (iv)  follows directly from Theorem \ref{thm:IV.2.3.}.
\end{proof}


\subsection{Review of Ueda's neighborhood theory}\label{section:review_uedatheory}

Let $X$ be a complex surface and $Y$ a compact curve with the topologically trivial normal bundle $N_{Y/X}$. 
Fix a finite open covering $\{U_j\}$ of $Y$. 
Since $Y$ is compact and K\"{a}hler,  $N_{Y/X}$ is $\textrm{U}(1)$-flat, i.e., 
the transition functions on $\{U_{jk}\}$ can be represented by $\textrm{U}(1)$-valued constant functions $\{t_{jk}\}$ ($\textrm{U}(1)=\{t\in\mathbb{C}\mid |t|=1\}$, see \cite[\S 1]{U}). 
Here $U_{jk}=U_{j}\cap U_{k}$. 
Take an open neighborhood $V_{j}$ of $U_{j}$ in $X$ and set $V:=\textstyle\bigcup_jV_j$. 
As shrinking $V_{j}$, we can choose the defining function $w_{j}$ of $U_{j}$ in $V_{j}$ such that $(w_j/w_k)|_{U_{jk}}\equiv t_{jk}$. 

For a system of such defining functions, 
the expansion of $t_{jk}w_k|_{V_{jk}}$ in the variable $w_j$ is written as 
\[
t_{jk}w_k=w_j+f_{jk}^{(n+1)}(z_j)\cdot w_j^{n+1}+O(w_j^{n+2})
\]
for $n\geq 1$. Such a system is said to be \textit{of type $n$}.
Then it follows that $\{(U_{jk}, f^{(n+1)}_{jk})\}$ satisfies the cocycle conditions (see \cite[\S 2]{U}). 
Consider the cohomology class 
\[
u_n(Y, X):=[\{(U_{jk}, f^{(n+1)}_{jk})\}]\in H^1(Y, N_{Y/X}^{-n}), 
\]
which is called \textit{the $n$-th Ueda class of $(Y,X)$}.
The $n$-th Ueda class is an obstruction to existence of a system of type $(n+1)$. 
Indeed, it is not difficult to see that 
a type $n$ system can be refined to be of type $(n+1)$ if (and only if) $u_n(Y, X)=0$. 
Therefore, the following two cases occur: 

\begin{itemize}
  \setlength{\parskip}{0cm} 
  \setlength{\itemsep}{0cm} 
\item 
There exists a positive integer $n$ such that the following holds:\\
For any $m\leq n$, there is a defining system of type $m$ such that 
$u_m(Y, X)=0$ for $m<n$ and $u_n(Y, X)\not=0$. 
\item 
For any positive integer $n$, there exists a defining system of type $n$ such that $u_n(Y, X)=0$. 
\end{itemize}
In the former case, the pair $(Y,X)$ is said to be of \textit{finite type} or \textit{of type ($\alpha$)}
(more precisely, \textit{of type $n$}). 
The latter case, we say, the pair $(Y,X)$ is \textit{infinite type}. 
For example, if $Y$ admits a \textit{holomorphic tubular neighborhood} in $X$, 
then $(Y,X)$ is infinite type. Here a holomorphic tubular neighborhood means 
a neighborhood of $Y$ in $X$ which is biholomorphic to that of the zero section of the normal bundle $N_{Y/X}$. 
More generally, consider the case where 
there exists a system $\{w_{j}\}$ as 
$$t_{jk}w_{k}=w_{j}.$$
Namely, the $\textrm{U}(1)$-flat structure on the normal bundle $N_{Y/X}$ can be extended to $[Y]$ around $Y$, 
where $[Y]$ is the line bundle which corresponds to the divisor $Y$, i.e., there exists a neighborhood $V$ of $Y$ in $X$ such that $[Y]|_V$ is $\textrm{U}(1)$-flat. 
In such a case, we say that $(Y,X)$ is {\it of type $(\beta)$}. 
Note that, in this case, $Y$ admits a \textit{pseudoflat} neighborhoods system in $X$, that is, a neighborhoods system with Levi-flat boundary. 
The remaining case is called type $(\gamma)$. 

\begin{remark}\label{rmk:finite_cover}
It does not change 
whether the type is finite or infinite 
after taking a finite covering space of a tubular neighborhood of $Y$, 
though the smallest number $n$ of non-vanishing Ueda classes varies. 
\end{remark}

Ueda showed the following result: 

\begin{theorem}\label{thm:uedathm3}{\rm (\cite[Theorem 3]{U})}
Suppose that the pair $(Y,X)$ is of infinite type. 
If the normal bundle $N_{Y/X}\in{\rm Pic}^0(Y)$ is torsion or satisfies the Diophantine condition (i.e. there exist positive constants $A$ and $\alpha$ such that $d(\mathbf{1}_{Y}, N_{Y/X}^{n})\geq A\cdot n^{-\alpha}$ holds for any positive integer $n$), 
then $Y$ is of type ($\beta$), that is, it admits a pseudoflat neighborhood system in $X$. 
\end{theorem}

Here $\mathbf{1}_{Y}$ is the holomorphically trivial line bundle on $Y$ and $d$ is the invariant distance on ${\rm Pic}^0(Y)$ defined by Ueda (see \cite[\S 4]{U}). 
As previous results, Arnol'd first studied neighborhoods of elliptic curves embedded 
in a surface with topologically trivial normal bundle (see \cite{A}). 
By regarding it as a kind of linearization problem, 
he applied the technique of Siegel's linearization theorem to this problem. 
Ueda's theorem is a partial generalization of Arnol'd theorem. 

When $(Y,X)$ is of type ($\alpha$), there are some results about 
the existence of strictly plurisubharmonic functions on a neighborhood of $Y$ 
and the constraint of its increasing degree. 

\begin{theorem}\label{thm:uedathm12}{\rm (\cite[Theorem 1, 2]{U})}
Suppose that the pair $(Y,X)$ is of type $n$. 
Then the following hold: 
\begin{enumerate}
\item[(i)] For any real number $a>n$, 
there exist a neighborhood $V$ of $Y$ and 
a strictly plurisubharmonic function $\Phi$ defined on $V\setminus Y$ 
such that $\Phi(p)=O({\rm dist}(p, Y)^{-a})$ as $p\to Y$. 
\item[(ii)] Let $V$ be a neighborhood of $Y$. 
For any positive real number $a<n$ and 
any plurisubharmonic function $\Psi$ defined on $V\setminus Y$ 
such that $\Psi(p)=o({\rm dist}(p, Y)^{-a})$ as $p\to Y$, 
there is a neighborhood $W$ of $Y$ in $V$ such that $\Psi|_{W\setminus Y}$ is constant. 
\end{enumerate}
\end{theorem}

\begin{remark}\label{rmk:beta_ph}
In contrast, 
by definition, the curve $Y$ of type $(\beta)$ admits a holomorphic foliation defined on an open neighborhood of $Y$ 
and  the holonomy along the compact leaf $Y$ is $\textrm{U}(1)$-linear.  
As is explained above, there exists a neighborhood $V$ of $Y$ in $X$ such that $[Y]|_V$ is $\textrm{U}(1)$-flat in this case. Thus, by considering the function $-\log |f_Y|_h^2$, where $f_Y\in H^0(V, [Y])$ is the canonical section and $h$ is a flat metric on $[Y]|_V$, 
there is a pluriharmonic function defined on $V\backslash Y$ which diverges logarithmically toward $Y$. 
\end{remark}

\begin{remark}\label{rmk:u1_flatness_ueda_class}
The Ueda type cannot be specified without the unitarity condition for the linear part of the holonomy, 
even if the holonomy is linearizable. 
In \cite[remark 2.2]{CLPT}, 
they constructed examples in the case III or VIII in \S 3 without the unitarity condition, 
although they are of type $(\alpha)$. 
\end{remark}


\section{Effect of holonomies on the Ueda types}\label{section:examples}

\subsection{Identification of the Ueda type divided by cases in accordance with dynamical properties of the holonomies}
Let $X$ be a complex surface and $Y$ be an elliptic curve embedded in $X$. 
Assume that there exist a non-singular holomorphic foliation $\mathcal{F}$ with $Y$ 
as a compact leaf, defined on a neighborhood of $Y$ in $X$. 
Take a point $p\in Y$ and generators $\gamma_1$ and $\gamma_2$ of $\pi_1(Y, p)$. 
We emphasize that these are fixed in the discussion below. 
Consider the holonomies $f:=h_{\mathcal{F}}(\gamma_1)$ and 
$g:=h_{\mathcal{F}}(\gamma_2)$ of $\mathcal{F}$ along $\gamma_{1}$ and $\gamma_{2}$ respectively, with respect to a transversal $\tau$ at $p$. 

According to the observation in \cite[Remark 2.2]{CLPT}, it is natural to focus on the case where 
both $\lambda=f'(0)$ and $\mu=g'(0)$ are elements of ${\rm U}(1)$ (see Remark \ref{rmk:u1_flatness_ueda_class}). 
For the fixed generators $\gamma_{1}$ and $\gamma_{2}$, the situation can be divided into ten cases I, II, \dots, X below in accordance with 
the (non-) torsionness of $\lambda$ and $\mu$ and the (non-) linearizability of $f$ and $g$. 
As the determining problem of Ueda's classification types $(\alpha)$, $(\beta)$, and $(\gamma)$ is stable under the change of the models by taking finite \`etale coverings (see Remark \ref{rmk:finite_cover}), we may assume that $\lambda$ (resp. $\mu$) is equal to $1$ if $\lambda$ (resp. $\mu$) is a torsion element of $\mathrm{U}(1)$ in what follows. 

\begin{description}
\item[Case I] $\lambda=\mu=1$, and $f$ and $g$ are linearizable. 
\item[Case II] $\lambda=\mu=1$, $f$ is linearizable, and $g$ is non-linearizable
\item[Case III] $\lambda=1$, $\mu$ is non-torsion, and both $f$ and $g$ are linearizable 
\item[Case IV] $\lambda=1$, $\mu$ is non-torsion, $f$ is linearizable, and $g$ is non-linearizable
\item[Case V] $\lambda=\mu=1$, and both $f$ and $g$ are non-linearizable
\item[Case VI] $\lambda=1$, $\mu$ is non-torsion, $f$ is non-linearizable, and $g$ is linearizable
\item[Case VII] $\lambda=1$, $\mu$ is non-torsion, and both $f$ and $g$ are non-linearizable
\item[Case VIII] both $\lambda$ and $\mu$ are non-torsion, and both $f$ and $g$ are linearizable
\item[Case IX] both $\lambda$ and $\mu$ are non-torsion, $f$ is linearizable, and $g$ is non-linearizable
\item[Case X] both $\lambda$ and $\mu$ are non-torsion and both $f$ and $g$ are non-linearizable
\end{description}

\begin{table}[htb]
\scalebox{0.8}{
  \begin{tabular}{|c|c||c|c|c|c|} \hline
    \multicolumn{2}{|c||}{} & \multicolumn{2}{c|}{$\mu\in \textrm{U}(1)$: torsion} & \multicolumn{2}{c|}{$\mu\in \textrm{U}(1)$: non-torsion}  \\ \cline{3-6}
     \multicolumn{2}{|c||}{}  & $g$: linearizable & $g$: non-linearizable &$g$: linearizable  &$g$: non-linearizable  \\ \hline\hline 
   $\lambda\in \textrm{U}(1)$  &$f$: linearizable  & I & II & III & IV \\ \cline{2-6}
   \ \ : torsion  & $f$: non-linearizable  &  & V & VI & VII \\ \hline
   $\lambda\in \textrm{U}(1)$ & $f$: linearizable &  &  & VIII & IX \\ \cline{2-6}
   \ \ : non-torsion &$f$: non-linearizable  &  &  &  & X \\ \hline
  \end{tabular}
}
\end{table}

\begin{remark}
For example in Case I, III and VIII, just the existence of the linearization maps $\varphi, \psi\in {\rm Diff}(\mathbb{C}, 0)$ such that $\varphi^{-1}$ $\circ$ $f\circ \varphi(w)=\lambda\cdot w$ and $\psi^{-1}$ $\circ$ $g$ $\circ$ $\psi(w)$ $=$ $\mu\cdot w$ is assumed and nothing on the relationship between $\varphi$ and $\psi$ is assumed literally. 
However in reality, it turns out that $f$ and $g$ can be linearized simultaneously in these cases (i.e. we have that $\varphi=\psi$, see \S 2). 
\end{remark}

It suffices to consider only these ten cases from the symmetry. 
Each case is invariant under conjugations, 
in particular, it does not depend on the choice of transversals. 
On the other hand, it {\it does} depend on the choice of generators of $\pi_{1}(Y,p)$. 
\begin{theorem}\label{thm:main_oldtype}
For the quadruple $(Y, X; \mathcal{F}, \{\gamma_1, \gamma_2\})$, the following statements hold. 
\begin{enumerate}
\item[$(i)$] In Case I, the pair $(Y, X)$ is of type ($\beta$). 
\item[$(ii)$] In Case II, the pair $(Y, X)$ is of type ($\alpha$).  
\item[$(iii)$] In Case III, the pair $(Y, X)$ is of type ($\beta$). 
\item[$(iv)$] In Case IV, the pair $(Y, X)$ is of type ($\gamma$).  
\item[$(v)$] In Case V, the pair $(Y, X)$ is of type ($\alpha$) or ($\beta$). 
Both pairs of types ($\alpha$) and ($\beta$) exist.   
\item[$(vi)$] No pairs $(Y, X)$ is in Case VI.  
\item[$(vii)$] No pairs $(Y, X)$ is in Case VII.  
\item[$(viii)$] In Case VIII, the pair $(Y, X)$ is of type ($\beta$).  
\item[$(ix)$] No pairs $(Y, X)$ is in Case IX. 
\item[$(x)$] In Case X, the pair $(Y, X)$ is of type ($\gamma$). 
\end{enumerate}
\end{theorem}

\subsection{Proof of Theorem \ref{thm:main_oldtype} except $(iv)$ and $(x)$}

Here we prove Theorem \ref{thm:main_oldtype} except the assertions $(iv)$ and $(x)$. 
The proof of $(iv)$ and $(x)$ are one of the main parts in the present paper. 
We describe the proof of them in \S 4. 
By using the generators $\{\gamma_1, \gamma_2\}$ of $\pi_1(Y, p)$, 
here we identify the elliptic curve $Y$ with the quotient $\mathbb{C}/\langle 1, \tau\rangle$ for a modulus $\tau\in \mathbb{H}:=\{z\in \mathbb{C}\mid {\rm Im}\,z>0\}$. 
In \S 3.2 and \S 3.3, we fix the generators of $\pi_{1}(Y,p)$  corresponding to the curves $c(t):=z+t$ and $c'(t):=z+\tau t$, and keep the notations $\gamma_{1}$ and $\gamma_{2}$ for them. 

First, we show the assertion $(i)$. 
Assume that $(Y, X; \mathcal{F}, \{\gamma_1, \gamma_2\})$ is in Case I. 
In this case, both $f$ and $g$ are the identity. 
By considering the foliation chart corresponding to this, we obtain a system $\{w_j\}$ of local (or even global)  defining functions of $Y$. 
Thus the pair $(Y, X)$ is of type $(\beta)$. 

Next, we show the assertion $(ii)$. Assume that $(Y, X; \mathcal{F}, \{\gamma_1, \gamma_2\})$ is in Case II. 
In this case, $f$ is the identity. 
Note that $N_{Y/X}$ is holomorphically trivial in this case. 
Let 
\[
g(w)=w+\sum_{\nu=2}^\infty b_\nu\cdot w^\nu
\]
be the expansion of $g$. 
Denote by $n$ the minimum element of the set $\{\nu\in\mathbb{Z}\mid \nu\geq 2,\ b_\nu\not=0\}$. 
Then, the foliation chart of $\mathcal{F}$ gives a system $\{w_j\}$ of local defining functions of type $n-1$, 
which means that the pair $(Y, X)$ is of type greater than or equal to $n-1$. 
By definition, the $(n-1)$-th Ueda class corresponds to (the conjugacy class of) the representation 
\[
\rho\colon \pi_1(Y, p)\cong\mathbb{Z}\oplus\mathbb{Z}\tau\to \mathbb{C}
\]
 defined by $\rho(1)=0$ 
and $\rho(\tau)=b_n$ under the natural identification $H^1(Y, N_{Y/X}^{-n+1})=H^1(Y, \mathcal{O}_Y)=H^{0, 1}(Y, \mathbb{C})$ 
and the injection $H^{0, 1}(Y, \mathbb{C})\to H^1(Y, \mathbb{C})$. 
Thus, we have that $u_{n-1}(Y, X)\not=0$, which means that the pair $(Y, X)$ is of type $n-1$. 
Therefore, the pair $(Y,X)$ is of type $(\alpha)$. 

The assertion $(iii)$ is shown by the same argument as the proof of $(i)$ above. 
The proof of $(v)$ is given in \S3.3. 
The assertions $(vi), (vii), (viii)$, and $(ix)$ follow from Proposition \ref{prop:comm}. 
\qed


\subsection{Examples}

Here we give some examples. 

\subsubsection{Examples of the case V}
We describe examples of the case V to show the assertion $(v)$ of Theorem \ref{thm:main_oldtype}. 
As $N_{Y/X}$ is torsion in this case, we can apply Theorem \ref{thm:uedathm3} to conclude that the type of the pair is whether ($\alpha$) or ($\beta$). 
By considering the example of \cite[Remark 2.2]{CLPT} with general choice of the representation $c_\bullet$, we obtain an example of the pair of type ($\alpha$) in Case V. 
In what follows, we construct a pair of type ($\beta$) in Case V. 
Define an affine bundle $V\to Y$ over $Y=\mathbb{C}/\langle 1, \tau\rangle$ by 
\[
V=\mathbb{C}^2/\left\langle
\left(
    \begin{array}{c}
      1 \\
      1 
    \end{array}
  \right), 
\left(
    \begin{array}{c}
      \tau \\
      \tau 
    \end{array}
  \right)\right\rangle, 
\]
or equivalently, $V$ is the quotient $\mathbb{C}^2/\sim$ of $\mathbb{C}^2$ with coordinates $(x, \xi)$ by the relation generated by $(x, \xi)\sim (x+1, \xi+1)\sim (x+\tau, \xi+\tau)$ (see also \S \ref{section:example_susp_constr}). 
The projection to $Y$ is the one induced by the first projection $(x, \xi)\mapsto x$. 
Let $X$ be the ruled surface over $Y$ which is a compactification of $V$ by adding the infinity section. 
We will identify the infinity section with $Y$ by the natural manner and also denote it by $Y$, i.e. $X=V\cup Y$. 
Denote by $\mathcal{F}$ the foliation on $X$ whose leaves are locally defined by the equation $\xi=(\text{constant})$. 
By regarding $w=1/\xi$ as a local defining function of $Y$, the holonomies $f=h_{\mathcal{F}}(\gamma_1)$ and $g=h_{\mathcal{F}}(\gamma_2)$ can be expressed as
\[
f(w)=\frac{1}{\frac{1}{w}-1}=\frac{w}{1-w}
\]
and 
\[
g(w)=\frac{1}{\frac{1}{w}-\tau}=\frac{w}{1-\tau w}.
\]
It follows from the direct calculation that this example $(Y, X; \mathcal{F}, \{\gamma_1, \gamma_2\})$ is in Case V. 
On the other hand, 
by considering another coordinate $(\widehat{x}, \widehat{\xi})$ of $\mathbb{C}^2$ defined by $\widehat{x}=x$ and $\widehat{\xi}=\xi-x$, 
we can easily see that $X$ is biholomorphic to $Y\times \mathbb{C}P^1$ 
and that $Y$ corresponds to the subvariety $Y\times \{\infty\}$ of $Y\times \mathbb{C}P^1$. 
Thus, we conclude that the pair is of type ($\beta$). 

\subsubsection{Suspension construction}\label{section:example_susp_constr}
We consider the suspension construction over $Y=\mathbb{C}_{x}/\langle 1, \tau\rangle$. 
Let $f,g$ be elements in $\textrm{Diff}(\mathbb{C},0)$ which satisfies $f\circ g=g\circ f$ and $U$ an admissible domain for both of $f$ and $g$.  
Then, we take the quotient space of $\mathbb{C}_{x}\times \mathbb{C}_{\xi}$ by the equivalence relation generated by $(x,\xi)\sim (x+1, f^{-1}(\xi))\sim (x+\tau, g^{-1}(\xi))$ for any $x\in \mathbb{C}_{x}$ and $\xi\in U\subset \mathbb{C}_{\xi}$. 
Denote by $\pi$ the projection to $Y$ induced by the first projection $(x,\xi)\mapsto x$. 
Then, we can choose a smooth tubular neighborhood $X_{0}$ of $Y$ with respect to $\pi$. 
This equips the holomorphic foliation $\mathcal{F}_{0}$ induced by $\{\xi=\text{constant}\}$ which has $Y$ as a compact leaf. 
The space $X_{0}$ is not necessarily invariant by the leaves. 
The holonomies of $\mathcal{F}_{0}$ along $\gamma_{1}$ (resp. $\gamma_{2}$) is  given by $f$ (resp. $g$) with respect to the transversal $\pi^{-1}(\{0\})$. 
This model automatically satisfies the condition (1) which appeared before Theorem \ref{thm:main}. 

As $C^{\infty}$ smooth foliations, the isomorphism classes of foliation germs along the compact leaf $Y$ is  determined by the conjugacy classes of holonomy homomorphisms (see e.g. \cite[Theorem 2.3.9]{CC}). 
Or equivalently, a foliation germ along the compact leaf $Y$ is isomorphic to that of $(X_{0},Y;\mathcal{F}_{0})$ whose holonomy is conjugate to the former one. 
However, such a statement does not hold in the holomorphic setting. 
In \cite{K}, the first author investigated examples of $(X, Y; \mathcal{F})$ which satisfy the conditions $(1)$ and $(2)$ which appeared in Theorem \ref{thm:main}.
Under these conditions, 
the isomorphism classes of \textit{holomorphic} foliation germs along $Y$ can be determined by information of the holonomies (see \cite[\S 2.2]{K} for example). 

At the end of this section, we describe some examples obtained by using the suspension construction. 
We restrict ourselves to the case that $f$ is the identity map (the condition (2)). 
First, if $g$ is linearizable, then $(X_{0}, Y; \mathcal{F}_{0}, \{\gamma_1, \gamma_2\})$ belongs to Case I or III. 
Next, we provide two examples of the case where $g$ is non-linearizable. 
If $g$ is defined as
\[
g(w)=\frac{w}{1-w},
\]
$(X_{0}, Y; \mathcal{F}_{0}, \{\gamma_1, \gamma_2\})$ gives a typical example of Case II, which is known as \textit{Serre's example}.
See e.g. \cite[Example 4.1]{K} for the detail. 
In \cite[\S 5.4]{U}, Ueda constructed the first example of type $(\gamma)$. 
This can be explained as the case that $g$ is a (monic) polynomial of degree $d\geq 2$ such that $\mu$ is a non-torsion element of ${\rm U}(1)$ which satisfies ``the strong Cremer condition" $\liminf_{n\to\infty} A^n\cdot |1-\mu^n|^{1/(d^n-1)}=0$. 
Ueda showed that, for each neighborhood $\Omega$ of the origin, there exists a periodic cycle 
\[
\{w, g(w), g^2(w), \dots, g^{\ell-1}(w), g^\ell(w)=w\}
\]
which is included in $\Omega$ for such a polynomial $g$ ({\it the existence of small cycles}). 
As a fixed point with Siegel disk never have such a property, it is clear that such $g$ is non-linearizable. 
Therefore, it follows that Ueda's example belongs to Case IV. 


\section{Proof of main theorems} 

\subsection{Proof of Theorem \ref{thm_ph}} 
Let $c$ be a closed curve in $Y$ as in Theorem \ref{thm_maincor}. 
Take a transversal $\tau$ of $\mathcal{F}$ with a fixed parametrization and set $f:=h_{\mathcal{F}}(\gamma_{1}) \in \textrm{Diff}(\mathbb{C},0)$ where $\gamma_{1}=\langle c\rangle \in \pi_{1}(Y,p)$. 
By the assumption, the holonomy $f$ is irrationally indifferent and non-linearizable. 
As we mentioned in \S1, these properties do not depend on the choice of $\tau$. 
Denote by $g:=h_{\mathcal{F}}(\gamma_2)$ where $\pi_1(Y, p) \cong\mathbb{Z}\gamma_1  \oplus \mathbb{Z} \gamma_2$.
Since $f$ and $g$ commute, it follows from Proposition \ref{prop:comm} (2) (applying by switching the roles of $f$ and $g$) that the holonomy $g$ is either 
\begin{enumerate}
\item[(i)] linearizable and rationally indifferent, or 
\item[(ii)] non-linearizable and irrationally indifferent. 
\end{enumerate}

Here we remark that, in the case (i), by changing the coordinate on $\tau$ and taking a finite covering space of $X$, one may read the following proof by assuming that $g$ is the identity. 

In both cases, there exists a complete invariant set $K$ (hedgehog) in $\tau$ under the action of the subgroup $\Gamma$ of $\textrm{Diff}(\tau,p)\cong \textrm{Diff}(\mathbb{C},0)$ generated by $f$ and $g$. 
The existence is guaranteed by Theorem \ref{thm:hedgehog}. 
As replacing a smaller transversal $\tau$, we can get an arbitrary small invariant set $K$ in $\tau$. 
We may assume that the saturated set 
\[
\textstyle \mathcal{F}(K):=\bigcup_{x\in K} L_{x}
\]
of $K$ is included in a relatively compact subset $W$ of $V$, where $L_{x}$ is the leaf of $\mathcal{F}$ through $x$. 
The set $\mathcal{F}(K)$ includes $Y$. 
By Theorem \ref{thm:hedgehog} (iii) and (iv), there is a leaf $L_{x_{0}}$ such that 
\[
\overline{L_{x_{0}}}=\mathcal{F}(\partial K)\supset Y, 
\]
where $\textstyle \mathcal{F}(\partial K):=\textstyle\bigcup_{x\in \partial K} L_{x}$. 

For sufficiently small $K$, the conformal type of leaves in $\mathcal{F}(K)$ is determined by the following lemma, 
which was taught by Professor Tetsuo Ueda. 

\begin{lemma}\label{lem_parabolic}
For sufficiently small $K$, any leaf $L$ of $\mathcal{F}(K)$ is parabolic, i.e. the universal covering space $\widetilde{L}$ is biholomorphic to $\mathbb{C}$. 
\end{lemma}
\begin{proof}[Proof of Lemma \ref{lem_parabolic}] 
By replacing $W$ with a smaller neighborhood of $Y$,
we can choose a smooth retraction $\phi \colon W\to Y$ whose restriction $\phi|_{L}\colon L\to Y$ on each leaf $L$ is an orientation preserving local diffeomorphism. 
If necessary, we here retake a smaller transversal $\tau$ to achieve that $\mathcal{F}(K)$ is included in $W$. 
Then, the dilatation 
\[
D_{\mathcal{F},\phi}(p):=\frac{\phi_{z}(p)+\phi_{\overline{z}}(p)}{\phi_{z}(p)-\phi_{\overline{z}}(p)} \in [1,\infty]
\]
of $\phi$ along leaves is defined, where $z$ is a complex coordinate on the leaf direction.  
(More precisely, we fix a holomorphic universal covering map $\mathbb{C}\to Y$ and consider the derivatives $\widetilde{\phi}_z$ and $\widetilde{\phi}_{\overline{z}}$ of a lift $\widetilde{\phi}\colon W\to \mathbb{C}$ of $\phi$.)  
Note that the derivatives satisfy $\phi_{z}(p)\to 1$ and $\phi_{\overline{z}}(p)\to 0$ as $p\to Y$. 
Since $Y$ is compact, the dilatation is bounded on $\mathcal{F}(K)$. 
Thus, the restriction $\phi|_{L}\colon L\to Y$ on a leaf $L\in \mathcal{F}(K)$ is a quasiconformal.
 
The universal covering space $\widetilde{L}$ is biholomorphic to either $\mathbb{D}$ or $\mathbb{C}$. 
The dilatation of a lift $\widetilde{\phi|_{L}}\colon \widetilde{L}\to \widetilde{Y}=\mathbb{C}$ is also bounded, so that it is a quasiconformal diffeomorphism. 
However, there are no quasiconformal homeomorphisms from $\mathbb{D}$ to $\mathbb{C}$ (see e.g. \cite{N}). 
Therefore, it follows that $\widetilde{L}$ is biholomorphic to $\mathbb{C}$. 
\end{proof} 
Let $\Phi\colon V\to \mathbb{R}\cup \{\infty\}$ be a continuous function which is pluriharmonic on $V\setminus Y$ and $i\colon L_{x_{0}}\to$ $\mathcal{F}(K)$ $\subset W \,(\Subset V)$ the inclusion, which is a holomorphic immersion. 
Then, the pullback $\varphi:=(\pi\circ i)^*\Phi$ of $\Phi$ is a harmonic function on $\widetilde{L_{x_{0}}}\cong \mathbb{C}$, where $\pi$ is the universal covering map $\pi\colon \widetilde{L_{x_{0}}}\to L_{x_{0}}$. 
Since $\varphi$ is bounded from below, Liouville's theorem shows that $\varphi$ is constant. 
Thus, $\Phi$ is constant on $\overline{L_{x_{0}}}=\mathcal{F}(\partial K)\supset Y$. 
By the continuity of $\Phi$ around $Y$, $\Phi$ is bounded from above on a neighborhood of $Y$. 
\qed

\subsection{Proof of Theorem \ref{thm_maincor}}
Let $Y$ be an elliptic curve embedded in a complex surface $X$ and $\mathcal{F}$ a holomorphic foliation as in Theorem \ref{thm_maincor}, defined on a neighborhood $V$ of $Y$. 
By the Serre duality, for $m\in \mathbb{Z}$, 
\[
H^{1}(Y,N_{Y/V}^{-m})\cong 
\left\{
\begin{array}{lll}
\mathbb{C} && (\textrm{if}~~N_{Y/V}^{-m}=\mathbf{1}_{Y})\\
0 && (\textrm{if}~~N_{Y/V}^{-m}\neq \mathbf{1}_{Y})
\end{array}
\right.,
\]
where we denote by $\mathbf{1}_{Y}$ the holomorphically trivial line bundle. 
It follows from the holonomy assumption and Proposition \ref{prop:comm} $(2)$ that $N_{Y/V}$ is non-torsion, i.e. $N_{Y/V}^{-m}\neq \mathbf{1}_{Y}$ for any $m>0$ 
(Consider a similar argument in the proof of the assertion $(ii)$ of Theorem \ref{thm:main_oldtype}. 
Here, we use Proposition \ref{prop:comm} $(2)$ for assuring that the representation $\pi_1(Y, p)\to \mathbb{C}$ which corresponds to the normal bundle is a unitary representation). 
Thus, $Y$ is either of type $(\beta)$ or of type $(\gamma)$. 

By assuming that $(Y,X)$ is of type $(\beta)$, we lead to contradiction. 
For an arbitrary small neighborhood of $Y$, say $V$ again, there exists a pluriharmonic function $\Phi\colon V\setminus Y\to \mathbb{R}$ such that $\Phi(p)\geq -c\log {\rm dist}(p, Y)$ as $p\to Y$ for some positive constant $c$ (see Remark \ref{rmk:beta_ph}). 
In particular, this is bounded from below on a compact neighborhood $V'\subset V$ of $Y$. 
Then, Theorem \ref{thm_ph} shows that $\Phi$ is bounded from above on $V'\supset Y$. 
This contradicts the growth condition of $\Phi$. 
\qed

\subsection{Proof of Theorem \ref{thm:main_oldtype}}
The assertions $(iv)$ and $(x)$ follows from Theorem \ref{thm_maincor}. 
See \S 3.2 for the proof of the others. 
\qed

\subsection{Proof of Theorem \ref{thm:main}}
As any $g$ can be classified into only one of the assertions $(i), (ii)$, or $(iii)$, ``only if" part follows from ``if" part. Therefore here we only show ``if"-part of each of the assertions. 

The case where the modulus of $\mu:=g'(0)$ is not equal to $1$, the assertion is shown in \cite{K} (Indeed, it is shown that the pair is of type $(\beta)$ in this case by using the assumption $(1)$). 
Therefore we may assume that $|\mu|=1$, and so the situation is as in \S 3, in what follows. 

First, assume that $\mu$ is a torsion element of $\mathrm{U}(1)$. 
If $g$ is of finite order, i.e. $g^n = \textrm{id}$ for some $n>0$, the situation is reduced to Case I in \S 3 by taking a finite covering space of a tubular neighborhood of $Y$. 
In this case, it follows from Theorem \ref{thm:main_oldtype} $(i)$ that the pair $(Y, X)$ is of type $(\beta)$. 
If $g$ is not of finite order, then it is non-linearizable at $0$.
Namely, it is in Case II. 
Thus, it follows from Theorem \ref{thm:main_oldtype} $(ii)$ (and the similar argument as above) that the pair is of type $(\alpha)$. 

Next, assume that $\mu$ is a non-torsion element of $\mathrm{U}(1)$. 
If $g$ is linearizable at $0$, the situation is in Case III. 
In this case, the pair is of type $(\beta)$ by Theorem \ref{thm:main_oldtype} $(iii)$. 
If $g$ is non-linearizable at $0$, the situation is in Case IV. 
It follows from Theorem \ref{thm:main_oldtype} $(iv)$ that the pair $(Y, X)$ is of type $(\gamma)$. 
\qed

\section{Proof of Theorem \ref{thm_psh}.}
Let $(Y, X; \mathcal{F})$ be as in Theorem \ref{thm_psh}. 
Consider a continuous function $\Phi\colon V \to \mathbb{R}\cup\{\infty\}$, defined on a neighborhood $V$ of $Y$, which is plurisubharmonic on $V\setminus Y$.  
Set
\[
\Omega:=\{p\in Y\mid \Phi(p)<\infty\}.
\] 
First, we show that $\Omega \neq \emptyset$ by assuming a contradiction. 
Second, we check that $\Omega=Y$ and $\Phi$ is extended as a plurisubharmonic function. 

We assume that $\Omega=\emptyset$ holds. 
Take a transversal $\tau$ of $\mathcal{F}$ at $p_{0}\in Y$. 
Let  $c$ be a closed curve as in Theorem \ref{thm_maincor} and $c'$ as in the assumption of Theorem \ref{thm_psh}. 
After taking a finite covering space if necessary, we may assume that $c$ and $c'$ generate the fundamental group $\pi_{1}(Y, p_{0})$. 
Note that the conclusion is not changed by this operation. 
Set $f:=h_{\mathcal{F}}(\langle c \rangle)$ and $g:=h_{\mathcal{F}}(\langle c' \rangle)$, 
where $f$ is irrationally indifferent and non-linearizable and $g$ is of finite order. 
By taking a finite covering space if necessary again, we may further assume that $g$ is the identity. 
As in the proof of Theorem \ref{thm_ph}, there exists a hedgehog $K$ of an admissible neighborhood $U$ in $\tau$ which is completely invariant under the action of $\Gamma =\langle f,g\rangle=\langle f\rangle$. 
Also as in the proof of Theorem \ref{thm_ph}, take a relatively compact neighborhood $W$ of $Y$ in $V$ such that the saturated set $\mathcal{F}(K)$ is included in $W$. 
By the assumption, the set
\[
V_{M}:=\{p\in V\mid \Phi(p)>M\},
\] 
is a neighborhood of $Y$ in $V$, where $M$ is a constant. 
By enlarging $M$, we may assume that $V_{M}\Subset W$ and $U_{M}:=V_{M}\cap \tau \Subset U$. 
Since $K\cap \partial U\neq \emptyset$ (see Theorem \ref{siegel}), $K\setminus  \overline{U_{M}} \neq \emptyset$. 

There exists a point $x_{0}$ of $\partial K$ such that the forward orbit 
\[
O^{+}_{f}(x_{0}) :=\{ f^{n}(x_{0}) \mid n\in \mathbb{N}\}
\] 
of $f$ is dense in $\partial K$. 
Indeed, it follows from Theorem \ref{thm:hedgehog} $(iii), (iv)$ (applying by switching the roles of $f$ and $g$) and the recurrence property of orbits in $K$ (see \cite[Corollaire 1]{P3}). 
By the holonomy condition, the leaf $L_{x_{0}}$ through $x_{0}$ is diffeomorphic to an open annulus. 

Choose a point $y_{0}\in K\cap \partial U$ (note $y_{0}\notin\overline{U_M}$) and a sufficiently small neighborhood $D$ of $y_{0}$ in $\tau$ such that $\widehat{D}$ is included in $\textrm{Int} (W\setminus V_{M})$. 
Here $\widehat{D}$ is the set defined as follows: 
\[
\widehat{D}:=\overline{\{(z, w)\mid z\in c', w\in D\}}, 
\]  
where we used coordinates $(z, w)$ in a Stein neighborhood of $Y\setminus \{q_{0}\}$ in $X$ where the point $q_{0}\in Y$ does not lie on the closed smooth curve $c'$, $z$ is a local coordinate on $Y\setminus \{q_0\}$, and each leaf of $\mathcal{F}$ is defined by $\{w={\rm constant}\}$ there (Or one may replace $Y\setminus\{q_0\}$ with any open neighborhood of $c'$ in $Y$). 
The existence of such a neighborhood is guaranteed by Siu's theorem \cite{Siu} (Note that it is shown in \cite{Siu} that any Stein submanifold admits a neighborhood which is biholomorphic to a neighborhood of the zero section of the normal bundle. As any holomorphic line bundle over an open Riemann surface is holomorphically trivial, $Y\setminus \{q_{0}\}$ admits a neighborhood which is biholomorphic to the product of $Y\setminus \{q_{0}\}$ and a disk). 
Since $0$ and $y_{0}$ is contained in $\overline{O^{+}_{f}(x_{0})}$, we can take positive integers $n_1<n_2<n_3$ such that 
\[
f^{n_1}(x_{0}), f^{n_3}(x_{0})\in D ~\textrm{and}~ f^{n_2}(x_{0})\in U_M.
\] 
We here take the lift of the closed curve $c$ to $L_{x_{0}}$ through $f^{n_1}(x_{0})$, denoted by $\widetilde{c}$. 
Three points $f^{n_1}(x_{0}), f^{n_2}(x_{0}), f^{n_3}(x_{0})$ lie on the path $\widetilde{c}$ in this order with respect to the natural orientation.  
Also, take the lift of the closed curve $c'$ to $L_{x_{0}}$ through $f^{n_j}(x_{0})$, $j=1,3$, which is expressed as
\[
\widetilde{c_{j}}'=\{(z, w)\mid z\in c', w=f^{n_j}(x_{0})\}. 
\]
It is included in $\widehat{D}$. 
Denote by $A$ the subset of $L_{x_{0}}$ bounded by $\widetilde{c_{1}}'$ and $\widetilde{c_{3}}'$. 
It follows from the above construction that $f^{n_2}(x_{0})$ lies on $\textrm{Int}\,A$. 
See figure \ref{fig:two} below. 

Let $i\colon L_{x_{0}}\to W$ be the inclusion, which is a holomorphic immersion. 
Then, the function $i^*\Phi$ is subharmonic and which satisfies $i^*\Phi|_{\partial A}<M$ by the construction. 
However, since $f^{n_2}(x_{0})\in V_M$, we obtain $\textstyle\sup_{A} i^{\ast}\Phi |_{A}>M$. This contradicts to the maximum principle. 

\begin{figure}[htpb]
 \begin{center}
  \hspace{-10mm}
  \includegraphics[width=85mm]{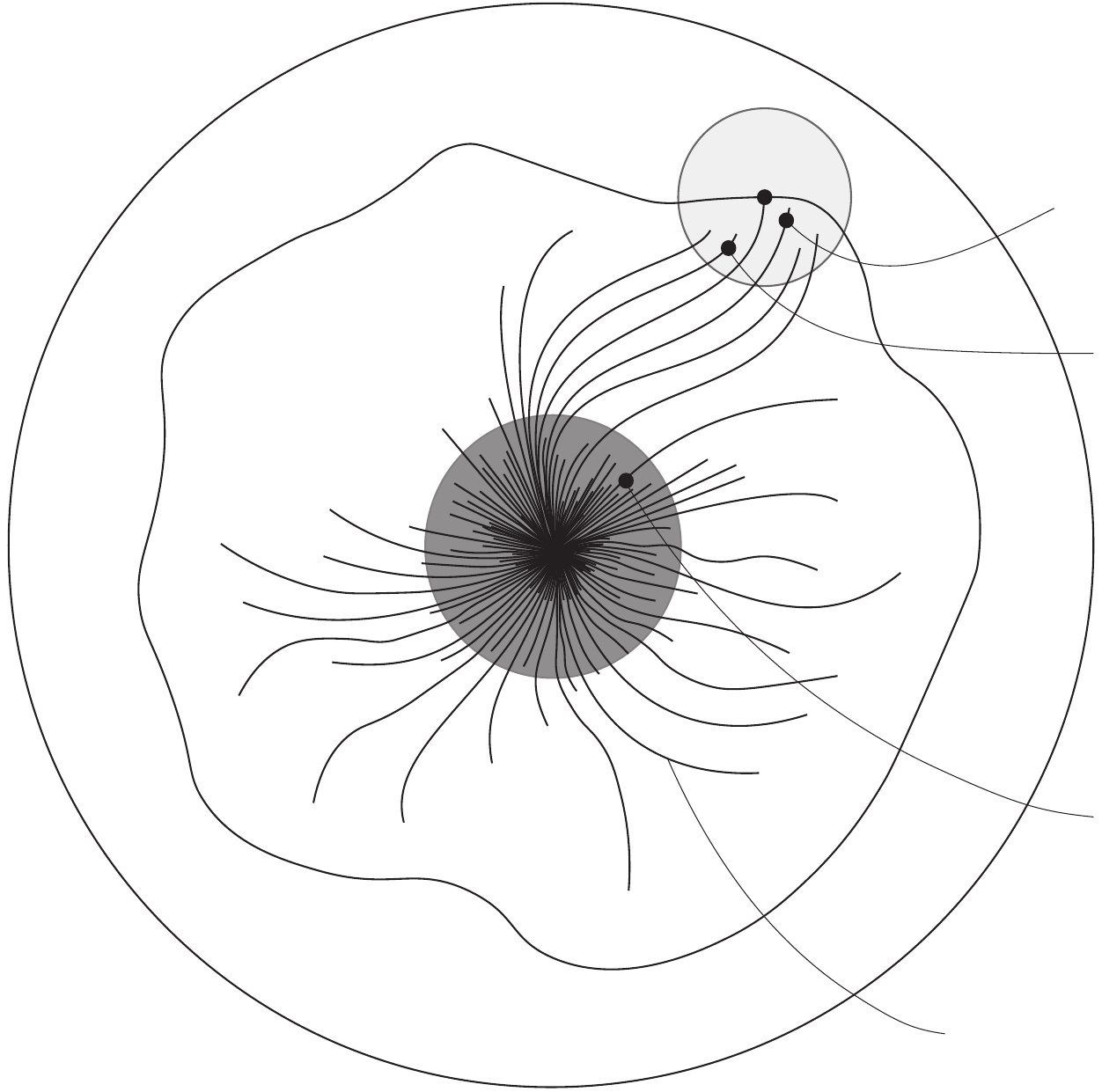}
  \hspace{-5mm}
  \includegraphics[width=85mm]{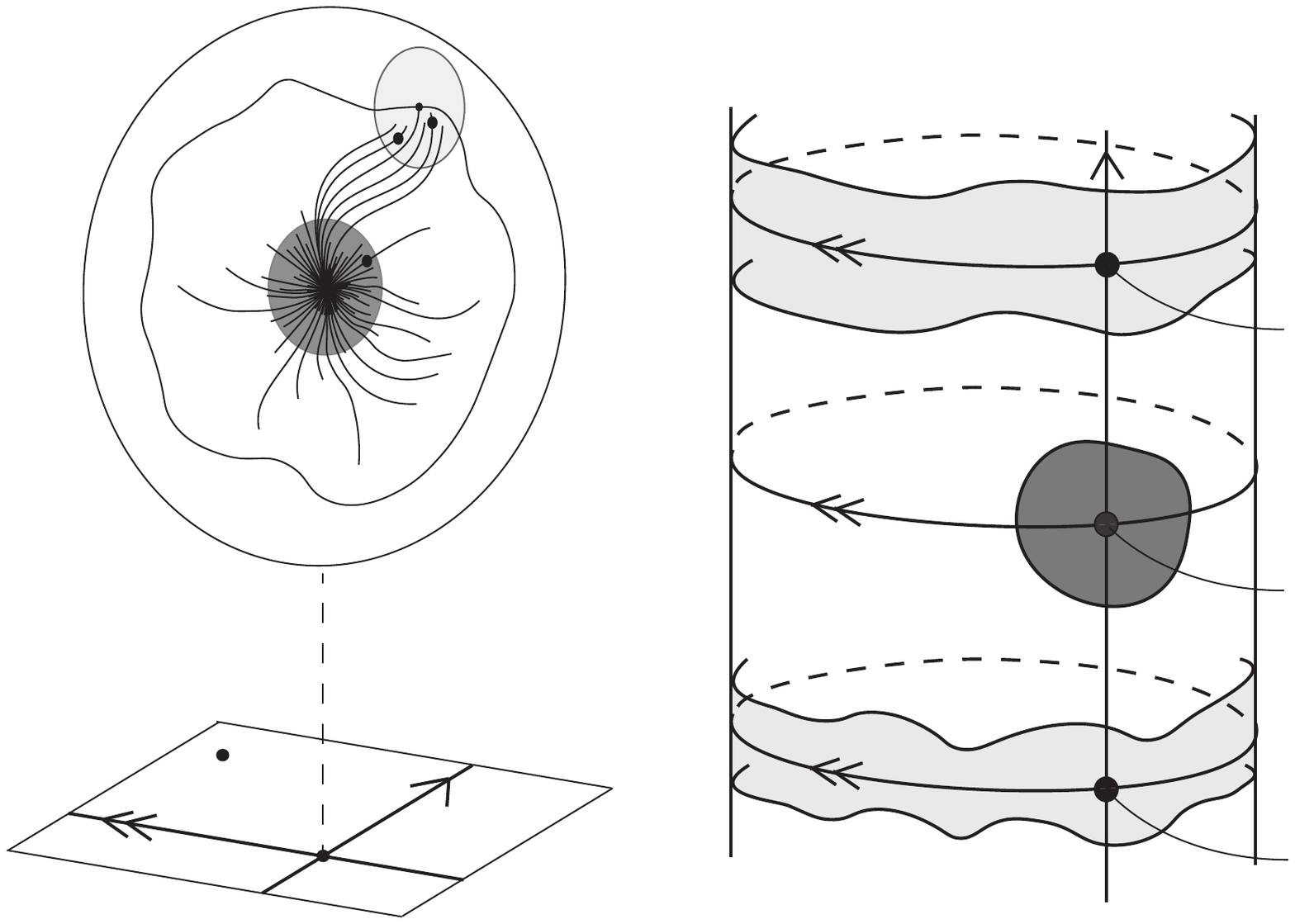}
 \end{center}
 \begin{picture}(0,0)
 \put(-114,170){\small$D$}
 \put(-95,164){\small$y_{0}$}
 \put(-48,155){\small$f^{n_{3}}(x_{0})$}
 \put(-42,133){\small$f^{n_{1}}(x_{0})$}
 \put(-43,58){\small$f^{n_{2}}(x_{0})$}
 \put(-202,132){\small$U$}
 \put(-15,25){\small$Y$}
  \put(-65,25){\small$K$}
 \put(-205,35){\small$\tau$}
 \put(-164,116){\small$U_{M}$}
 \put(47,30){\small$p_{0}$}
\put(105,58){\small$\widetilde{c_{1}}'$}
\put(105,105){\small$\widetilde{c_{2}}'$}
\put(105,150){\small$\widetilde{c_{3}}'$}
\put(30,23){\small$c$}
\put(77,29){\small$c'$}
 \put(28,48){\small$q_{0}$}
 \put(210,175){\small$L_{x_{0}}$}
 \put(218,127){\small$f^{n_{3}}(x_{0})$}
 \put(218,80){\small$f^{n_{2}}(x_{0})$}
 \put(218,32){\small$f^{n_{1}}(x_{0})$}
 \put(178,23){\small$\widetilde{c}$}
\end{picture}
\vspace{-5mm}
 \caption{The hedgehog $K$ on $\tau$ which is invariant under the holonomy along $Y$ is depicted in the left figure. There is a point $x_{0}\in \partial K=K$ whose forward orbit is dense in $\partial K$.  
 The right figure is the corresponding dense leaf $L_{x_{0}}$ in $\mathcal{F}$ which is an open annulus. 
 The function $\Phi$ takes values less than $M$ on an open neighborhood of $\widetilde{c_{1}}'\cup \widetilde{c_{3}}'$ (colored in light gray) and greater than $M$ on an open neighborhood of $f^{n_{2}}(x_{0})$ (colored in dark gray). 
 \label{fig:two}
 }
\end{figure}

Next, we show that $\Omega=Y$.
By \cite[Theorem 5.24]{D}, $\Phi$ is plurisubharmonic on a neighborhood of $p$ in $V$ for each $p\in \Omega$. 
Since $Y$ is connected and $\Omega$ is a non-vacuous open subset of $Y$, it suffices to show that $\Omega$ is closed. 
Take a point $q\in \overline{\Omega}$ 
and a coordinate $(z, w)$ around $q$ in $V$ so that $(z, w)=(0, 0)$ at $q$ and $Y=\{w=0\}$ on the locus. 
Consider a sequence $\{q_\nu=(z_\nu, 0)\}_\nu\subset \Omega$ which tends to $q$ as $\nu\to\infty$. 
Then, for a sufficiently small $\varepsilon>0$,
\[
\Phi(q)
=\Phi(0, 0)
=\lim_{\nu\to \infty}\Phi(z_\nu, 0)
\leq\lim_{\nu\to \infty}\frac{1}{2\pi}\int_{0}^{2\pi}\Phi(z_\nu, \varepsilon e^{\sqrt{-1}\theta})\,d\theta. 
\]
For a fixed $\epsilon>0$, the sequence $\{\textstyle \max_{0\leq \theta<2\pi}\Phi(z_\nu, \varepsilon e^{\sqrt{-1}\theta})\}_\nu$ is bounded from above, so that we have $q\in \Omega$, i.e. $\Omega$ is closed. Therefore, $\Omega=Y$ holds, and the proof is complete.
\qed


\section{Proof of Corollary \ref{cor:semipositivity}}

By \cite{K2}, $L$ is not semi-positive when the pair $(Y, X)$ is of type $(\alpha)$. 
Assume that the pair $(Y, X)$ is of type $(\beta)$. 
Then there exists a neighborhood $V$ of $Y$ such that $L$ admits a unitary flat metric $h_V$ on a neighborhood $V$ of $Y$ 
(i.e. $h_V$ is a $C^\infty$ Hermitian metric on $L|_V$ whose Chern curvature is $0$, see \S \ref{section:review_uedatheory}). 
On the other hand, $L$ admits a singular Hermitian metric $h_{\rm sing}$ such that $h_{\rm sing}|_{X\setminus Y}$ is a $C^\infty$ Hermitian metric on $L|_{X\setminus Y}$, $h_{\rm sing}\to\infty$ holds when a point approaches to $Y$, and that the Chern curvature of $h_{\rm sing}|_{X\setminus Y}$ is $0$. 
Indeed, the singular Hermitian metric defined by $|f_Y|_{h_{\rm sing}}^2\equiv 1$ satisfies this property, where $f_Y\in H^0(X, L)$ is a section with ${\rm div}(f_Y)=Y$. 
A $C^\infty$ Hermitian metric $h$ on $L$ with semi-positive curvature can be constructed by using {\it the regularized minimum construction} for these two metrics $h_V$ and $h_{\rm sing}$, which is the same construction as we used for proving \cite[Corollary 3.4]{K1}. 
This proves the semi-positivity of $L$ when the pair $(Y, X)$ is of type $(\beta)$. 

Therefore all we have to do is to show that $L$ is not semi-positive assuming that the triple $(Y, X,\mathcal{F})$ is in Case IV, which is done by the same manner as in the proof of the main theorem in \cite{K2} by using Theorem \ref{thm_psh} instead of \cite[Theorem 2]{U}. 
\qed

In the cases we described in \S 3, Case X is the most interesting case from the viewpoint of Conjecture \ref{conjecture:main}: 

\begin{question}
Does $L$ admit a $C^\infty$ Hermitian metric with semi-positive curvature when the pair $(Y, X)$ is in Case X? 
\end{question}

\appendix
\section{hedgehogs for commuting holomorphic diffeomorphisms}\label{section:simplified_prf_of_Thm.III.14}

In this appendix, we give a proof of Theorem \ref{Thm.III.14}, based on \cite{P5}. 
His proof requires some facts which are only written in the unpublished paper (mainly, the uniqueness result of hedgehogs {\cite[Thm.III.4]{P5}}). 
Theorem \ref{Thm.III.14} is a slightly weaker version of his original statement, which is however enough for our purpose. 
As an advantage, we can prove it while avoiding the use of such facts. 
More precisely, we prove it by showing Lemma \ref{lem_a} instead of {\cite[Thm.III.4]{P5}}. 
We here describe the outline of the proof. 
Proposition \ref{propII3} below is obtained in \cite{P3} and \cite{P5}, which is a key proposition. 
See the references and also \cite{Y} for the proof.  
Theorem \ref{pm12} is also shown, which is one of the main applications of Proposition \ref{propII3}. 
We give a sketch of the proof. 
Furthermore,  we prepare Lemma \ref{hdf} and Lemma \ref{lem_a}. 
Finally, Theorem \ref{Thm.III.14} will be proved by using Lemma \ref{lem_a} and Theorem \ref{pm12}. 
In this section, we consider $\mathbb{C}P^1$ with the Fubini-Study metric $g_{FS}$ 
as the ambient space.

\begin{proposition}[{\cite[Proposition 1]{P3}, \cite[Proposition II.3]{P5}}]\label{propII3} 
Let $g(z)=\mu z +O(z^2)$ be a local holomorphic diffeomorphism with the irrationally indifferent fixed point $0$. 
Assume that $g$ is non-linearizable at $0$. 
Let $U$ be an admissible domain of $g$ and 
$K$ a hedgehog of $(U,g)$.
Then, 
for each $n\in \mathbb{N}$, 
there exists a quintuple $(\Omega_{n}, B_{n}, \eta_{n}, R_{n}, A_{n})$ associated with $K$, 
where $\Omega_{n}$ is an open neighborhood of $K$ in $\mathbb{C}P^1$, 
$B_{n}$ is a closed annulus in $\Omega_{n} \setminus K$ separating $\partial \Omega_{n}$ from $K$, 
$\eta_{n}$ is a Jordan closed curve in the interior of $B_{n}$ separating two boundary components of $B_{n}$, 
$R_{n}$ is a closed quadrilateral in $B_{n}$,  and 
$A_{n}$ is a closed annulus in $\Omega_{n}\setminus (K\cup R_{n})$ separating $R_{n}$ from $K$
whose modulus tends to $\infty$ as $n\to \infty$. 
The quintuple satisfies the following conditions:
\begin{itemize} 
\item[(1)]
For each $q_{j}~(j=0,\dots,n)$, the iterations $g^{\pm q_{j}}$ are defined on $\Omega_{n}$, 
where $q_{n}$ is given by 
the continued fractional approximation $(p_{n}/q_{n})_{n\in \mathbb{N}}$ of the irrational number $\alpha=(\log{\mu})/2\pi\sqrt{-1}$. 
\item[(2)] 
For any point $z$ in $\eta_{n}$, 
there exists an iteration $g^{m_{n}}(z)$ which is contained in $R_{n}$.  
\item[(3)] 
For any point $z$ in the component of $\mathbb{C}P^1\setminus B_{n}$, 
if there is an integer $k$ such that 
$g^{k}(z)$ is contained in the other component of $\mathbb{C}P^1\setminus B_{n}$,  
then there exists an iteration $g^{k_{n}}(z)$ which is contained in $R_{n}$. 
\end{itemize}
\end{proposition}

\begin{figure}[htpb]
 \begin{center}
  \includegraphics[width=76mm]{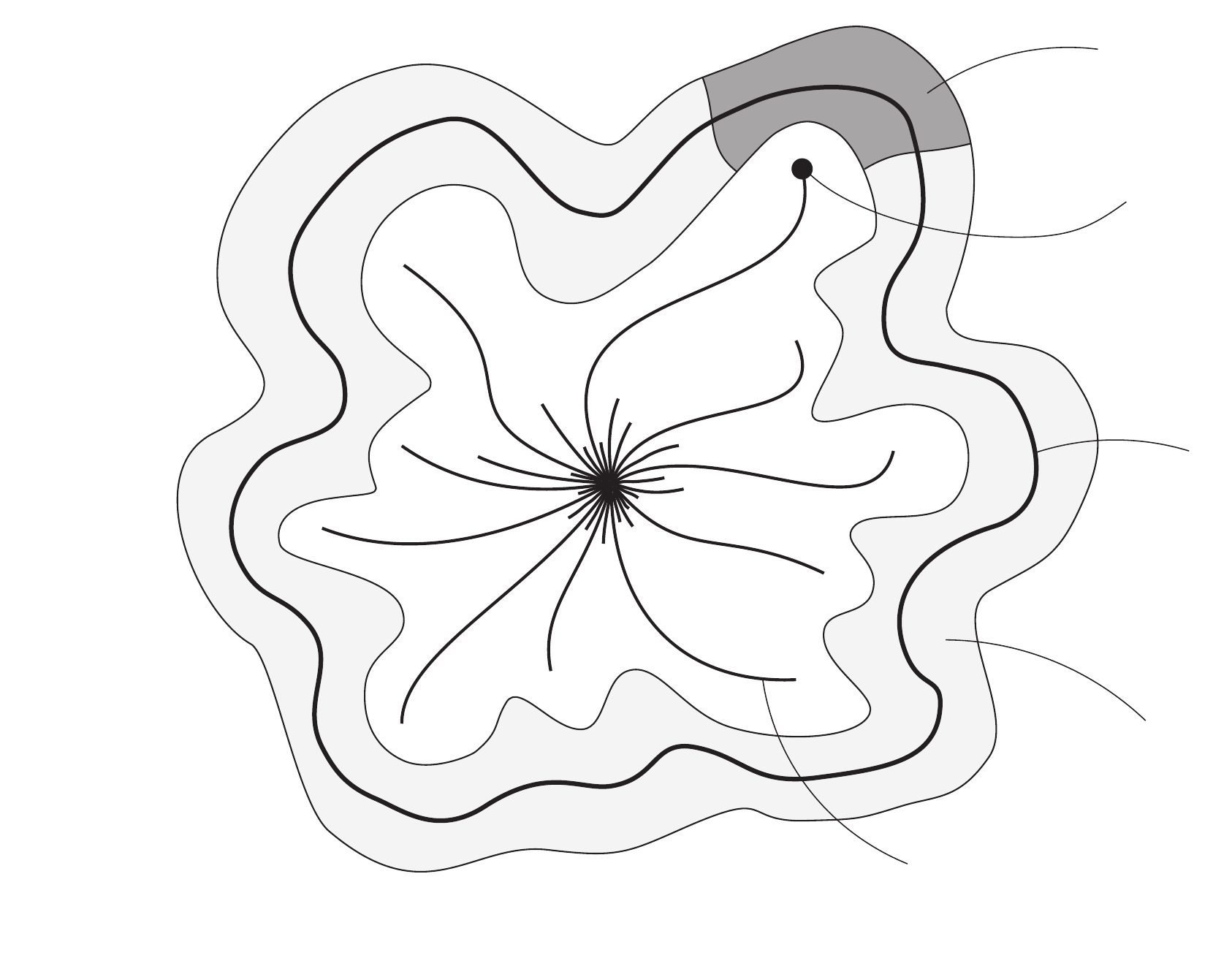}
  \hspace{5mm}
  \includegraphics[width=76mm]{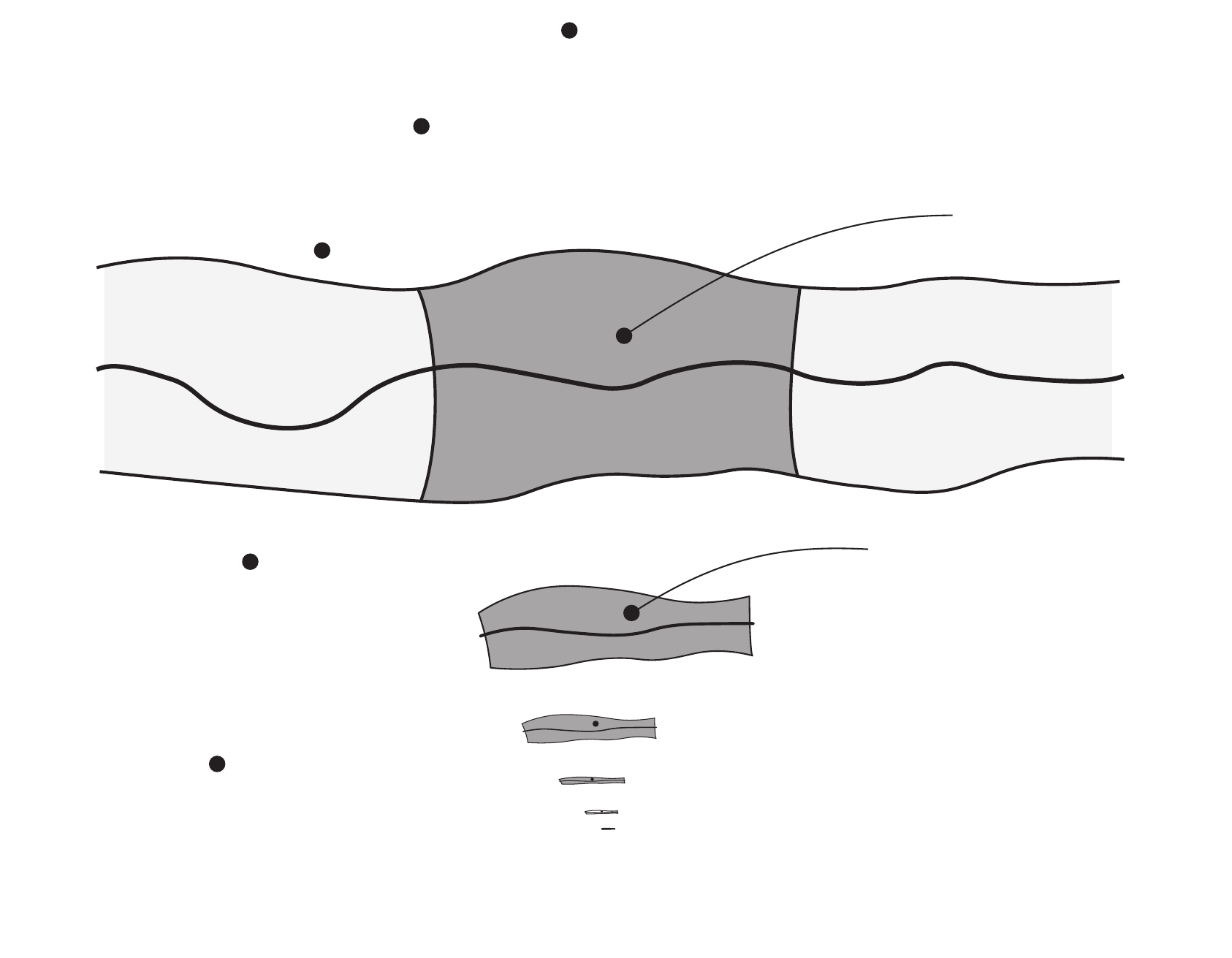}
 \end{center}
 \begin{picture}(0,0)
 \put(-30,173){\small$R_{n}$}
 \put(-25,150){\small$z_{0}$}
 \put(-15,102){\small$\eta_{n}$}
 \put(-25,50){\small$B_{n}$}
  \put(-66,28){\small$K$}
 \put(-205,35){\small$\Omega_{n}$}
 
 \put(13,118){\small$\eta_{n}$}
 \put(35,127){\small$B_{n}$}
 \put(102,147){\small$R_{n}$}
 \put(148,69){\small$R_{n+1}$}
\put(128,53){\small$R_{n+2}$}
\put(115,29){\small$z_{0}$}
\put(103,185){\small$z$}
 \put(28,78){\small$g^{k}(z)$}
 \put(183,147){\small$g^{k_{n}}(z)$}
 \put(168,87){\small$g^{k_{n+1}}(z)$}
\end{picture}
\vspace{-2mm}
 \caption{The separating annulus $B_{n}$, the meridian curve $\eta_{n}$ of  $B_{n}$, 
 and the quadrilateral $R_{n}$ associated with a hedgehog $K$ of $g$$\colon$ 
 The trapped subsequence $(g^{k_{n}}(z))_{n\in \mathbb{N}}$ in the statement (3) is depicted in the right figure.
 The quadrilateral $R_{n}$ converges to a point $z_{0}$ in $K\cap \partial U$. See Lemma \ref{hdf}. 
 }
 \label{fig:three}
\end{figure}

This is a rewriting of \cite[Proposition 1]{P3} as the statement for the hedgehog $K$ 
through a uniformization map $\psi\colon \mathbb{D}\to \mathbb{C}P^1\setminus K$. 
In fact, he used this version in \cite[\S 3 and \S 4]{P3}. 
Let us denote the corresponding quintuple in $\mathbb{D}$ 
by $(\widetilde{\Omega}_{n}, \widetilde{B}_{n}, \widetilde{\eta}_{n}, \widetilde{R}_{n}, \widetilde{A}_{n})_{n\in \mathbb{N}}$.
Note the symbols which we use. 
Compare with the original statement of \cite[Proposition 1]{P3}, 
the two boundary components of $\widetilde{B}_{n}$ correspond to curves $\gamma^{(n)}_{0}$ and $\gamma^{(n)}_{1}$, 
$\widetilde{\eta}_{n}$ corresponds to $\gamma^{(n)}$. Also, $\widetilde{R}_{n}$ 
and $\widetilde{A}_{n}$ correspond to $R_{n}$ and $A_{n}$ respectively. 

The statements $(2)$ and $(3)$ imply that the dynamics around the hedgehog behaves as ``quasi-rotation''. 
The closed curve $\eta_{n}$ is called a quasi-invariant curve, that is, 
after some iterations, the curve returns to near the initial position in the sense of Hausdorff distance 
with respect to the Poincar\'{e} metric on $\mathbb{C}P^{1}\setminus K$. 
For more details, see \cite{P5}, \cite{Y}, and \cite{P6}.

\begin{lemma}\label{hdf}
Let $(U,g)$ and $(\Omega_{n}, B_{n}, \eta_{n}, R_{n}, A_{n})_{n\in \mathbb{N}}$ associated with $K=K_{(U,g)}$ be as in Proposition \ref{propII3}.  
Then the annulus $B_{n}$ converges to $K$ 
in the sense of Hausdorff convergence with respect to the Fubini-Study metric. 
Moreover, for any point $z_{0}$ in $K\cap \partial U$, 
we can take the quadrilateral $R_{n}$ such that it converges to the point $z_{0}$. 
\end{lemma}

\begin{proof}[Proof of Lemma \ref{hdf}]

It is clear from the construction of $\widetilde{B}_{n}$ in the proof of \cite[Proposition II.3]{P5} or \cite{Y}
that the annulus $\widetilde{B}_{n}$ converges to the boundary $\partial \mathbb{D}$ and 
the corresponding annulus $B_{n}$ converges to the hedgehog $K$. 
Therefore, we show only the latter statement. 
Choose any point $z_{0}$ in $K\cap \partial U$. 
Since $\partial U$ is of class $C^1$, 
there is a path $\gamma\colon [0,1) \to \mathbb{C}P^1\setminus \overline{U}$ which lands at $z_{0}$, that is, 
the limit $\textstyle \lim_{t\to1}\gamma(t)$ exists and is $z_{0}$. 
For a uniformization map $\psi\colon \mathbb{D}\to \mathbb{C}P^1\setminus K$, 
it follows that the path $\gamma$ maps under $\psi^{-1}$ to a path $\widetilde{\gamma}$ in $\mathbb{D}$ which lands at some point on $\partial \mathbb{D}$ (see e.g. \cite[Corollary 17.10.]{M1}). 
The landing point is denoted by $p$. 
For sufficiently large $n$, the closed annulus $\widetilde{B}_{n}$ intersects with the path $\widetilde{\gamma}$. 
We can choose a point $q_{n}=\widetilde{\gamma}(t_{n})$ in $\widetilde{B}_{n}\cap \widetilde{\gamma}$ for each $n$
such that $(t_{n})_{n\in \mathbb{N}}$ is an increasing sequence. 
According to the construction of $\widetilde{R}_{n}$ in \cite{P5} or \cite{Y}, 
we can construct a quadrilateral $\widetilde{R}_{n}$ which contains the point $q_{n}$. 
Note that $q_{n}$ converges to $p$ as $n\to \infty$. 
Let us return to $\mathbb{C}P^1\setminus K$ under $\psi$. 
The quadrilaterals $\widetilde{R}_{n}$ and the path $\widetilde{\gamma}$ map to $R_{n}$ and $\gamma$. 
See Figure \ref{fig:four}. 
Then, we apply the modulus inequality (cf. \cite[\S 6.4]{LV})
$$\textrm{Mod}(A_{n})\leq \frac{2\pi^2}{\ell^2}$$ 
to the separating annulus $A_{n}$ surrounding $R_{n}$, 
where $\ell$ is the infimum of the length of closed curves separating two boundary components of $A_{n}$ 
with respect to the Fubini-Study metric $g_{FS}$ on $\mathbb{C}P^1$.
Since the modulus $\textrm{Mod}(A_{n})$ tends to $\infty$ from Proposition \ref{propII3}, 
it follows from the standard argument that 
the boundary beside $R_{n}$ degenerates to a single point. 
Hence, so does the quadrilateral $R_{n}$. 
By the choice of $R_{n}$'s, it converges to the point $z_{0}$. 
\end{proof}

\begin{figure}[htpb]

 \begin{center}
 \includegraphics[width=79mm]{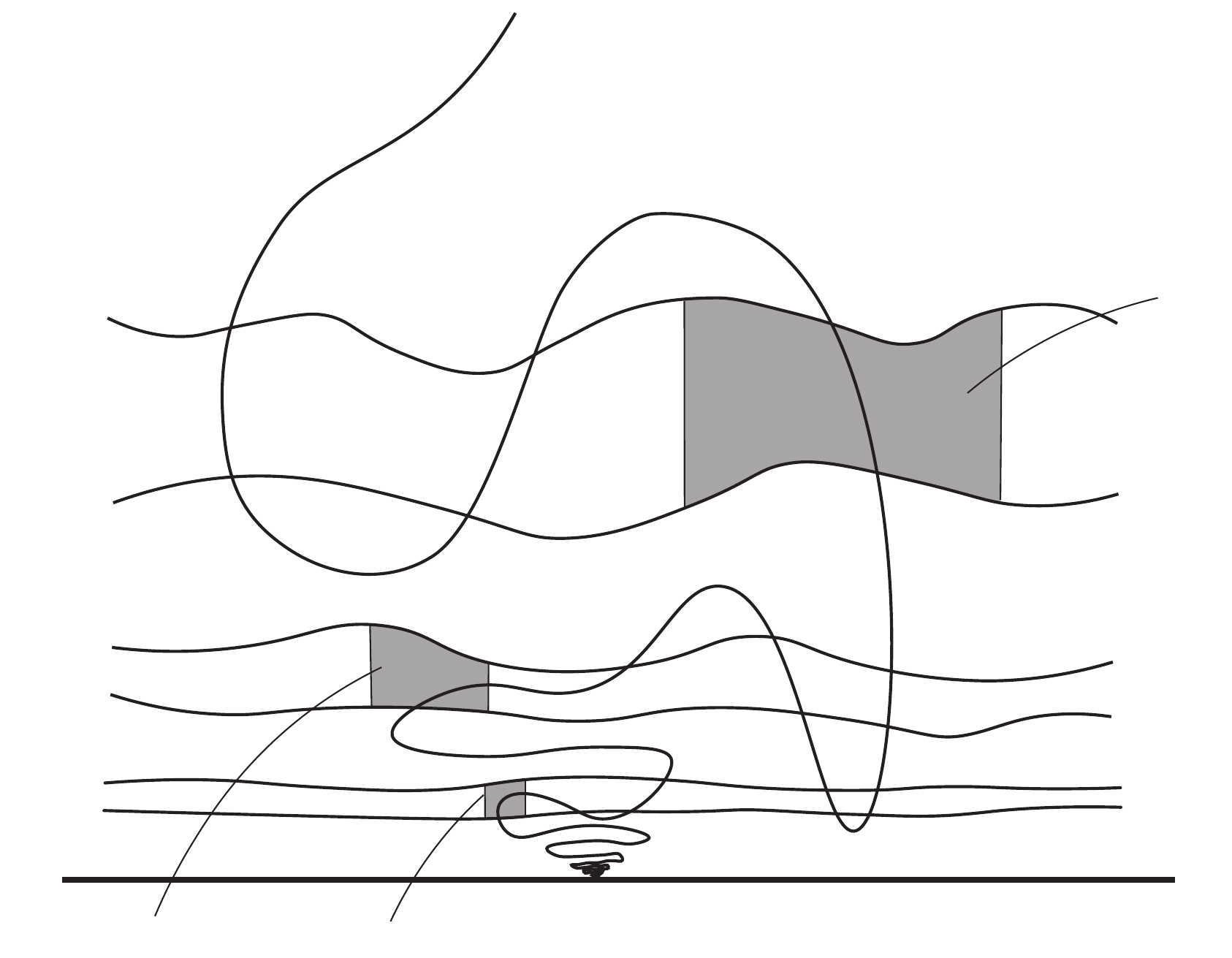}
  \hspace{-3mm}
  \includegraphics[width=78.8mm]{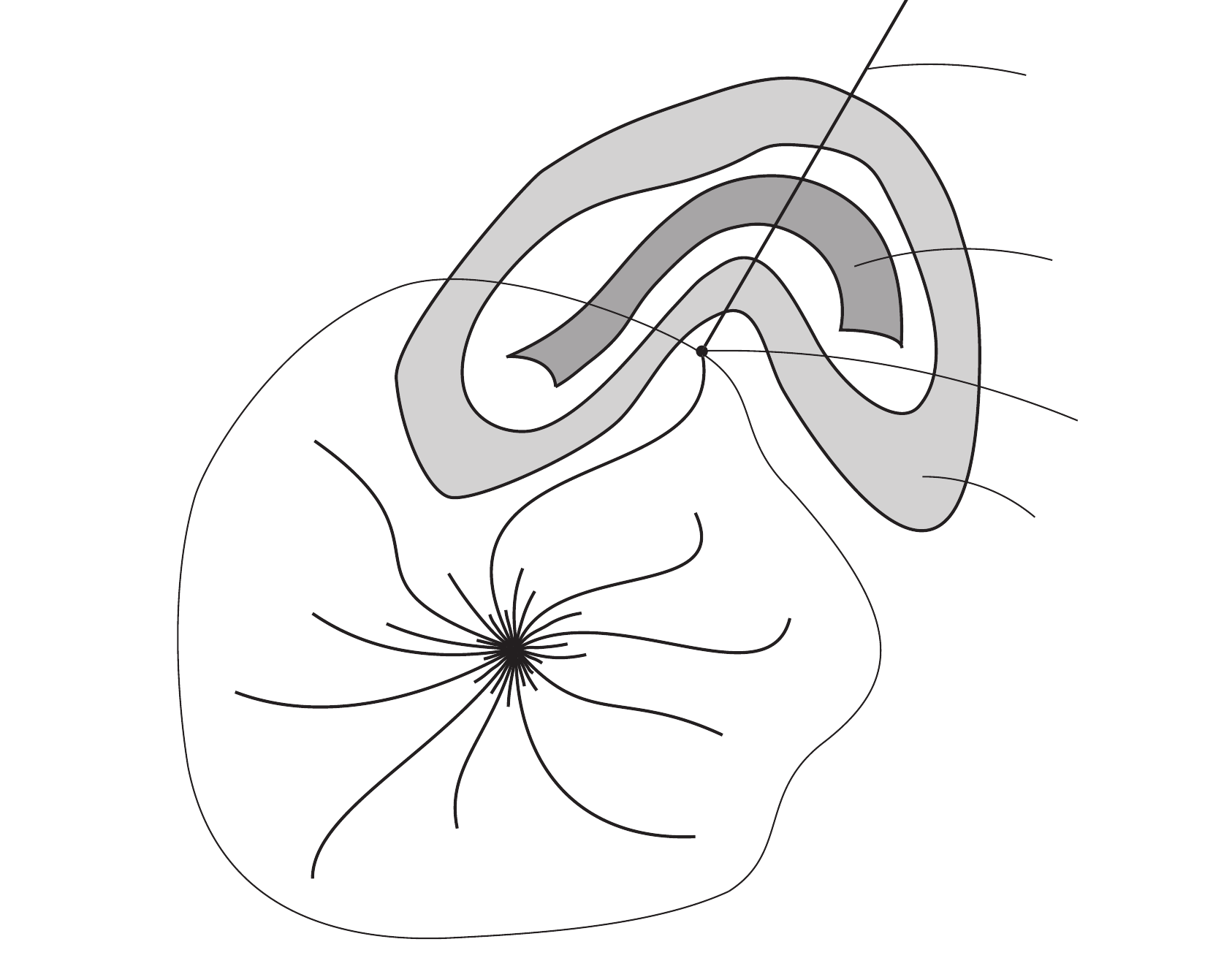}
 \end{center}
 \begin{picture}(0,0)
 \put(-160,175){\small$\widetilde{\gamma}$}
 \put(-10,135){\small$\widetilde{R}_{n}$}
 \put(-230,109){\small$\widetilde{B}_{n}$}
 \put(-235,65){\small$\widetilde{B}_{n+1}$}
 \put(-235,40){\small$\widetilde{B}_{n+2}$}
 \put(-210,10){\small$\widetilde{R}_{n+1}$}
 \put(-165,10){\small$\widetilde{R}_{n+2}$}
 \put(-5,27){\small$\partial \mathbb{D}$}
 
 \put(50,97){\small$K$}
 \put(150,30){\small$U$}
 \put(191,175){\small$\gamma$}
 \put(195,138){\small$R_{n}$}
 \put(199,110){\small$z_{0}$}
 \put(189,90){\small$A_{n}$}
\end{picture}

 \caption{There is a (sequence of) closed quadrilateral $R_{n}$ intersecting with the path $\gamma$, 
 which is surrounded by a closed annulus $A_{n}$ whose modulus tends to $\infty$ as $n\to \infty$. 
 The construction in $\mathbb{D}$ (under a uniformization map $\psi\colon \mathbb{D}\to \mathbb{C}P^1\setminus K$) 
 is depicted in the left figure. 
 }\label{fig:four}
\end{figure}

\begin{theorem}[{\cite[Theorem 1]{P3}}, {\cite[Theorem III.12.]{P5}}]\label{pm12}
Let $g(z)=\mu z+O(z^2)$ be a local holomorphic diffeomorphism
with the irrationally indifferent fixed point $0$.  
Assume that $g$ is non-linearizable. 
Then, 
the sequence $(g^{n}(z))_{n\in\mathbb{N}}$ does not converge to $0$ as $n\to \infty$ 
for any point $z$ distinct from $0$. 
\end{theorem}

Here, we give a sketch of the proof of Theorem \ref{pm12} following \cite{P5}. 
Let $U$ be an admissible domain for $g$ and $K$ a hedgehog of $(U,g)$. 
First, we take a point $z\in K\setminus \{0\}$.
It is known that any orbits in $K$ are recurrent (\cite[Corollaire 1]{P3}), 
so that the sequence $(g^{n}(z))_{n\in \mathbb{N}}$ does not converge to $0$. 
Second, take a point $z\not\in K$ where $g$ is defined and a point $z_{0}\in K\cap \partial U\neq \emptyset$. 
Let $B_{n}$ and $R_{n}$ be as in Proposition \ref{propII3} and Lemma \ref{hdf}.  
After enlarging $n$, 
we may assume that the point $z$ belongs to the component of $\mathbb{C}P^1\setminus B_{n}$ which does not contain $K$. 
If the sequence $(g^{k}(z))_{k\in\mathbb{N}}$ converges to $0$, 
then 
it follows from Proposition \ref{propII3} $(3)$ that 
there exists a subsequence  $(g^{k_{n}}(z))_{n\in\mathbb{N}}$ such that $g^{k_{n}}(z)\in R_{n}$ for each $n$. See Figure \ref{fig:three}. 
Therefore, by Lemma \ref{hdf}, the sequence $(g^{k}(z))_{k\in\mathbb{N}}$ accumulates at $z_{0}$. 
This contradicts the assumption, so that the statement follows. 

To show Theorem \ref{Thm.III.14}, we prepare the following lemma. 

\begin{lemma}\label{lem_a}
Let $g$ be as Proposition \ref{propII3}. 
Let $K$ and $K'$ be two hedgehogs of $(U,g)$ and $(U',g)$, 
where $U$ and $U'$ are admissible domains of $g$. 
Then, $K\subset K'$ or $K'\subset K$ hold. 
\end{lemma}

\begin{proof}[Proof of Lemma \ref{lem_a}]
We prove this by contradiction. 
Assume that $K\not\subset K'$ and $K'\not\subset K$ hold. 
Set $\mathbb{D}_{r}=\{z\in \mathbb{C}\mid |z|<r\}$, 
$\mathbb{D}=\mathbb{D}_{1}$, and $\mathbb{A}_{R}=\{z\in \mathbb{C}\mid R<|z|<1\}$.
First, we take a uniformization map 
$$\varphi\colon \mathbb{D} \to \mathbb{C}P^1 \setminus K'.$$ 
Consider an open neighborhood of $K'$ as 
$$V_{\epsilon}=\varphi({\mathbb{A}_{1-\epsilon}})\cup K'$$
for $\epsilon>0$.   
Since  $K\not\subset K'$ and $K$ is compact, there is $\delta>0$ such that 
$K\subset \overline{V_{\delta}}$ and $K\cap \partial V_{\delta}\neq \emptyset$.  
Choose a point $z_{0}\in K\cap \partial V_{\delta}$. 
Note that $z_{0}\not\in V_{\delta/2}$. 
We take the connected component of $U\cap V_{\delta}$ whose closure contains $K$. 
After a suitable smoothing of the boundary, the domain is an admissible domain of $g$, denoted by $V$. 
Note that $K \subset \overline{V}$ and $z_{0}\in K\cap \partial V\neq \emptyset$ still hold. 
Therefore, $K$ can be regarded as a hedgehog of $(V,g)$. 
Now apply Proposition \ref{propII3} and Lemma \ref{hdf} to the pair $(K=K_{(V,g)}, \{z_{0}\})$. 
For each $n\in \mathbb{N}$, take $\eta_{n}$ and $R_{n}$ as in the proposition. 

Let us show that there exists an integer $N$ such that $R_{n}\cap V_{\delta/2}=\emptyset$ and $\eta_{n}\cap K'\neq \emptyset$ hold for any $n\geq N$. 
The former statement follows directly from Lemma \ref{hdf}.  
On the other hand, 
the latter statement is shown by contradiction. 
Assume that there is a subsequence $(n_{k})_{k\in \mathbb{N}}$ so that $\eta_{n_{k}}\cap K' =\emptyset$.
Similarly to the previous paragraph, 
take a uniformization map 
$$\psi\colon \mathbb{D} \to \mathbb{C}P^1 \setminus K, $$
an open neighborhood of $K$ as 
$$V'_{\epsilon}=\psi({\mathbb{A}_{1-\epsilon}})\cup K$$
for $\epsilon>0$, 
and choose $\delta'>0$ such that 
$K'\subset \overline{V'_{\delta'}}$ and $K'\cap \partial V'_{\delta'}\neq \emptyset$ hold. 
Here we used the assumption $K'\not\subset K$ and the compactness of $K'$. 
Also, we choose a point $z'_{0}\in K'\cap V'_{\delta'}$. 
By Lemma \ref{hdf}, $\eta_{n_{k}}$ is contained in $V'_{\delta'/2}$ for large $k$. 
The Jordan curve theorem shows that $\eta_{n_{k}}$ decomposes $\mathbb{C}P^1$ 
into two domains $W_{0}$ and $W_{\infty}$ such that $\partial W_{0}=\partial W_{\infty}=\eta_{n_{k}}$, 
where $0\in W_{0}$ and $\infty \in W_{\infty}$ hold. Note that $z'_{0}$ belongs to $W_{\infty}$.  
By the assumption above, $K'$ is contained in $W_{0}\cup W_{\infty}$. 
However, $0\in K'\cap W_{0}\neq \emptyset$ and $z'_{0}\in K'\cap W_{\infty}\neq \emptyset$ hold. 
This contradicts the connectivity of $K'$, thus the latter statement follows. 

For sufficiently large $n$, take a point $z_{1}\in \eta_{n}\cap K'\neq \emptyset$. 
It follows from Proposition \ref{propII3} (2) that 
there exists an iteration $g^{m_{n}}(z_{1})$ contained in $R_{n}$. 
On the other hand, since $R_{n}\cap V_{\delta/2} = \emptyset$, 
the point $g^{m_{n}}(z_{1})$ lies outside of $V_{\delta/2}$. 
This contradicts the invariance of $K'$ under $g$. 
\end{proof}

\begin{proof}[Proof of Theorem \ref{Thm.III.14}.] 
For any open neighborhood $W$ of $0$, there is a small $R>0$ such that 
$\mathbb{D}_{R}=\{|z|<R\}\subset W$ and $f,g,g\circ f,$ and $f\circ g$ are defined on $\mathbb{D}_{R}$, 
further, $f$ and $g$ are univalent on an open neighborhood of $\overline{\mathbb{D}}_{R}$.
For sufficiently small $\epsilon>0$ and any $r\in (0,\epsilon)$, 
$f(\mathbb{D}_{r})\subset \mathbb{D}_{R}$ and $f^{-1}(\mathbb{D}_{r})\subset \mathbb{D}_{R}$. 
By Theorem \ref{siegel}, 
for each $r\in (0,\epsilon)$, there is a hedgehog $K_{r}$ of $(\mathbb{D}_{r},g)$.
Since $f$ and $g$ commute, $f(K_{r})$ is also a hedgehog of $(f(\mathbb{D}_{r}),g)$. 
As applying Lemma \ref{lem_a} to these hedgehogs,  
we have $K_{r}\subset f(K_{r})$ or $f(K_{r})\subset K_{r}$. 
We may assume that $f(K_{r})\subset K_{r}$ holds by exchanging $f$ and $f^{-1}$. 

The rest part is the same as the proof of {\cite[Thm.III.14]{P5}}.
Iterating by $f$, which is well-defined,  
we have the nested sequence
$$K_{r}\supset f(K_{r}) \supset f^2(K_{r}) \supset \cdots. $$
The set $\textstyle L=\underset{n\geqq 0}{\bigcap} f^{n}(K_{r})$ satisfies the following properties:
\begin{itemize}
\item[(1)] $L$ is compact, connected, and $\mathbb{C}\backslash L$ is connected, 
\item[(2)] $0\in L$,
\item[(3)] $L\neq \{0\}$, and 
\item[(4)] $L$ is invariant under $f,f^{-1},g$, and $g^{-1}$. 
\end{itemize}
It is not difficult to show (1),(2), and (4). 
Let us show (3). If $L=\{0\}$, then, for any $z\in K_{r}\setminus \{0\}$, 
the sequence $(f^{n}(z))_{n\in\mathbb{N}}$ converges to $0$ as $n\to \infty$.
However, $f$ is also non-linearizable from Proposition \ref{prop:comm}, 
so that this contradicts Theorem \ref{pm12}. 
Therefore 
$L$ is a common hedgehog of $f$ and $g$.
\end{proof}


\end{document}